\documentclass[oneside]{amsart}
\usepackage{placeins}
\usepackage{sseq,pgffor,sidenotes}
\usepackage[a4paper]{geometry}
\usepackage{amsmath,amsthm,amssymb,array,enumerate,graphicx,etoolbox,tabularx}
\usepackage[all]{xy}
\usepackage{mathtools}
\usepackage{nth}
\usepackage{todonotes}
\usepackage{xcolor} 
\usepackage[pagebackref]{hyperref}
\definecolor{dark-red}{rgb}{0.5,0.15,0.15}
\definecolor{dark-blue}{rgb}{0.15,0.15,0.6}
\definecolor{dark-green}{rgb}{0.15,0.6,0.15}
\hypersetup{
    colorlinks, linkcolor=dark-red,
    citecolor=dark-blue, urlcolor=dark-green
}	
\newcommand{\spec}{\mathrm{Spec}}
\renewcommand{\hom}{\mathrm{Hom}}
\newcommand{\circledbullet}{
\tikz{
\pgfsetbaselinepointlater{\pgfpointanchor{X}{base}}
\pgfcircle{\pgfpointorigin}{0.12cm}
\pgfusepath{stroke}
\node (X) {$\bullet$};
}}
\newcommand{\cbull}{ 
\text{\footnotesize \circledbullet}}
\def\biblio{\bibliography{../biblio}\bibliographystyle{alpha}}
\numberwithin{equation}{section}
\usepackage[nameinlink,capitalise,noabbrev]{cleveref}
\usepackage[justification=centering]{caption}

\newcount\tmp

\newcommand{\F}{\mathbb{F}}
\newcommand{\Z}{\mathbb{Z}}
\newcommand{\Q}{\mathbb{Q}}
\newcommand{\G}{\mathbb{G}}

\newcommand{\W}{\mathbb{W}}
\renewcommand{\S}{\mathbb{S}}
\newcommand{\xr}{\xrightarrow}
\renewcommand{\frak}{\mathfrak}
\newcommand{\einfty}{\mathbf{E}_\infty}

\newcommand{\cal}{\mathcal}

\newcommand{\pics}{\mathfrak{pic}}

\usepackage{comment}

\makeatletter
\newcommand*{\myfnsymbolsingle}[1]{%
  \ensuremath{%
    \ifcase#1
    \or 
      *%
    \or 
      \dagger
    \or 
      \ddagger
    \or 
      \mathsection
    \or 
      \mathparagraph
    \else 
      \@ctrerr  
    \fi
  }%
}   
\makeatother

\renewcommand*{\backref}[1]{}
\renewcommand*{\backrefalt}[4]{%
  \ifcase #1 %
No citations.
  \or
(cit. on p. #2).%
  \else
(cit on pp. #2).%
  \fi%
}

\usepackage{alphalph}
\newalphalph{\myfnsymbolmult}[mult]{\myfnsymbolsingle}{}

\DeclareMathOperator{\pic}{Pic}

\DeclareMathOperator{\Ext}{Ext}
\DeclareMathOperator{\Mod}{Mod}
\DeclareMathOperator{\End}{End}

\DeclareMathOperator{\Gal}{Gal}
\newcommand{\sym}{\mathrm{Sym}}
\newcommand{\gl}{\mathfrak{gl}}
\newcommand{\dequal}{\stackrel{\cdot}{=}}
\newcommand{\Pic}{\operatorname{\mathcal{P}ic}}

\newtheorem{theorem}{Theorem}[section]

\newtheorem{lemma}[theorem]{Lemma} 
\newtheorem{cor}[theorem]{Corollary}
\newtheorem{question}[theorem]{Question}

\newtheorem{prop}[theorem]{Proposition} \theoremstyle{definition}
\newtheorem{rem}[theorem]{Remark} 
\newtheorem{definition}[theorem]{Definition}
 \newtheorem{cons}[theorem]{Construction}
\newtheorem{example}[theorem]{Example}

\Crefname{cor}{Corollary}{Corollaries}
\Crefname{conjecture}{Conjecture}{Conjectures}
\Crefname{rem}{Remark}{Remarks}
\Crefname{prop}{Proposition}{Propositions}	
\Crefname{question}{Question}{Questions}
\Crefname{figure}{Figure}{Figures}

\title{Picard groups of higher real $K$-theory spectra at height $p-1$}
\author{Drew Heard}
\address{Max Planck Institute for Mathematics, Bonn}
\email{drew.heard@mpim-bonn.mpg.de}
\author{Akhil Mathew}
\address{Department of Mathematics, Harvard University, Cambridge, MA}
\email{amathew@math.harvard.edu}
\author{Vesna Stojanoska}
\address{Department of Mathematics, University of Illinois at Urbana-Champaign, Urbana, IL}
\email{vesna@illinois.edu}

\begin{document}

\begin{abstract}
Using the  descent spectral sequence for a Galois extension of ring spectra,
we compute the Picard group of the higher real $K$-theory spectra of Hopkins and Miller at height $n=p-1$, for $p$ an odd prime. More generally, we determine the Picard groups of the homotopy fixed points spectra $E_n^{hG}$, where $E_n$ is Lubin-Tate $E$-theory at the prime $p$ and height $n=p-1$, and $G$ is \emph{any} finite subgroup of the extended Morava stabilizer group. We find that these Picard groups are always cyclic, generated by the suspension. 
\end{abstract}
\maketitle
\setcounter{tocdepth}{1}
{\hypersetup{linkcolor=black}\tableofcontents}
\def\biblio{}
\section{Introduction}\label{sec:introduction}

\subsection{Picard groups in stable homotopy theory}
Let $(\cal{C},\otimes,\mathbf{1})$ be any symmetric monoidal category. Recall that an
object  $X \in \cal{C}$ is called \emph{invertible} if there exists an object
$X^{-1} \in \mathcal{C}$ such that $X \otimes X^{-1} \simeq \mathbf{1}$. We let $\pic(\cal{C})$ denote
the group of isomorphism classes of invertible objects. For example, if $\cal{C}
= \Mod_R^{cl}$ is the category of (classical) $R$-modules, for $R$ a
commutative ring, then this recovers the classical notion of the Picard group
of a ring. As another example, it is not hard to show that the Picard group of
the stable homotopy category is $\Z$, generated by $S^1$. 
Our goal is to develop tools based on descent theory for determining the Picard groups of
\emph{module categories} in stable homotopy theory, i.e., module spectra over
an $\einfty$-ring spectrum, and to apply them to a specific case of interest
in chromatic homotopy theory. 

Let $R$ be an $\einfty$-ring spectrum, and let $\Mod_R$ denote its symmetric
 monoidal $\infty$-category of modules (i.e., module spectra).
We are interested in the Picard group of $\Mod_R$ (or equivalently, of its
homotopy category). When the homotopy groups of $R$ 
are a \emph{regular} ring, it is often possible to reduce this problem to pure
algebra (i.e., the Picard group of the graded ring $\pi_*(R)$), cf.
\cite{baker_richter_invertible}, \cite[Thm.~2.4.6]{mathew_stojanoska_picard},
and 
\cite[Thm.~6.4]{HillMeier} for the most general statement. However, in general, the determination of the Picard
group is a subtle and complicated problem. 

Our basic setup is  that 
we will be working with ring spectra which themselves do not have
regular homotopy groups, but are nonetheless closely related to ring spectra that do and whose
Picard groups can be easily determined.
We refer to \cite{mathew_stojanoska_picard} for a more detailed introduction to
these types of questions. 

\newtheorem*{setup}{Basic problem}
\begin{setup}
Consider an $\einfty$-ring $B$ such that $\pi_*(B)$ is regular, together with an action 
of a finite group $G$ on $B$ as an $\einfty$-ring. From this, we can build a new $\einfty$-ring $A =
B^{hG}$ as the homotopy fixed points of the $G$-action.
As $\pi_*(B)$ was assumed regular, the Picard group of $B$ is the Picard
group of the graded ring $\pi_*(B)$. 
However, in general $\pi_*(A)$ will be very far from regular, so one expects the determination of the Picard group 
of $A$ to be an interesting \emph{homotopical} problem. 
\end{setup}

We will primarily be interested in the case where $A \to B$ is a $G$-Galois extension of $\einfty$-ring
spectra in the sense of Rognes \cite{rognes_galois}. That is, we assume there exists a $G$-action on $B$ as an $\einfty$-$A$-algebra such that $A \simeq B^{hG}$ and
the natural map $B \otimes_A B \to F(G_+,B)$ is an equivalence. We additionally
assume that the extension is \emph{faithful}, i.e. that for each $A$-module $N$ with
$N \otimes_A B \simeq \ast$, we have $N \simeq \ast$.
In this case, unpublished work of Gepner-Lawson gives a powerful tool for computing the
Picard group of $A$ via a derived form of Galois descent. For an account of this, the reader is referred to \cite{mathew_stojanoska_picard}.
In this paper, we will use related techniques to show that in a certain class of examples, one does not
obtain new ``exotic'' classes in $\pic(A)$. 

\subsection{The descent spectral sequence}

To describe the Gepner-Lawson method, we will need to use
higher homotopy coherences which we now review (again, see
\cite{mathew_stojanoska_picard} for a more detailed exposition). 
Let $(\mathcal{C}, \otimes, \mathbf{1})$ be a symmetric monoidal
$\infty$-category.  Keeping track of all higher isomorphisms, we obtain a 
\emph{Picard $\infty$-groupoid} (or \emph{Picard space}), $\Pic(\cal{C})$ of invertible objects in
$\mathcal{C}$, whose connected components give the Picard group of
$\mathcal{C}$. 

Thanks to the symmetric monoidal
product $\otimes$, $\Pic(\mathcal{C})$ is a grouplike $\einfty$-space, and so is the delooping
of a connective spectrum $\pics(\cal{C})$ called the \emph{Picard spectrum}. We have
identifications
\[
\begin{split}
\pi_0\Pic(\cal{C}) &\simeq \pic(\cal{C}),\\
\pi_1\Pic(\cal{C}) &\simeq (\pi_0\End(\mathbf{1}))^\times, \text{ and }\\
\pi_k\Pic(\cal{C}) &\simeq \pi_{k-1}\End(\mathbf{1}),\text{ for } k > 1. 
\end{split}
\]
\begin{example}
	Suppose that $\cal{C} = \Mod_R^{cl} $ is the symmetric monoidal category of (discrete) modules over a
	commutative ring  $R$. Then $\Pic(\Mod_R^{cl})$ is a space with
\[
\begin{split}
\pi_0\Pic(\Mod_R^{cl}) &\simeq \pic(R),\\
\pi_1\Pic(\Mod_R^{cl}) &\simeq R^\times, \text{ and}\\
\pi_k\Pic(\Mod_R^{cl}) & = 0,\text{ for } k>1. 
\end{split}
\]

\end{example}
 Let $R$ be an $\einfty$-ring. We write $\Pic(R)$ and $\pics(R)$ for $\Pic(\Mod_R)$ and $\pics(\Mod_R)$ respectively. 
The space $\Pic(R)$ (or the spectrum $\pics(R)$) provides a means of capturing
the higher homotopy coherences that appear in descent processes. 
In particular, we have the following
basic descent result due in unpublished work to Gepner--Lawson \cite{GepnerLawson}. 
\begin{theorem}[{Cf.~\cite[Sec. 3.3]{mathew_stojanoska_picard}}]
\label{galoispicdesc}
 Suppose that $A \to B$ is a faithful $G$-Galois extension of $\einfty$-rings. Then there is an equivalence of connective spectra
 \[
\pics(A) \simeq \tau_{\ge 0} \pics(B)^{hG}.
 \]
In the associated homotopy fixed point spectral sequence
 \begin{equation}\label{eq:picss}
H^s(G,\pi_t \pics(B)) \Rightarrow \pi_{t-s}(\pics B)^{hG},
 \end{equation}
 the abutment for  $t=s$ is the Picard group $\pic(A)$. 
\end{theorem}
Gepner--Lawson, in unpublished
work \cite{GepnerLawson}, have completely calculated this spectral sequence in the
fundamental example of
the $C_2$-Galois extension $KO \to KU$, and have shown that $\pic(KO) \simeq
\Z/8$, generated by $\Sigma KO$. This calculation, via a slightly different method, is originally due to Hopkins
(unpublished). We refer to \cite{mathew_stojanoska_picard} for an account of the  calculation of this spectral sequence.  
In practice, the hope in applying \Cref{galoispicdesc} above is that $\pic(B)$ is easy to determine because the
homotopy groups $\pi_*(B)$ are somehow simpler, and so the problem revolves
around determining the behavior of the spectral sequence.

In fact, the above spectral sequence \eqref{eq:picss} turns out to be very calculable in specific
instances by comparing with the additive spectral sequence
\begin{equation} \label{eq:addhfpss} E_2^{s,t}   = H^s( G, \pi_t B) \implies
\pi_{t-s } A, \end{equation}
arising from the equivalence $A \simeq B^{hG}$.
In practice, one already knows the behavior of the additive spectral sequence
in computing $\pi_*(A)$. 
Since for $k>1$ there is a $G$-equivariant isomorphism $\pi_k(\pics(B)) \simeq
\pi_{k-1}B$, one might hope that there is a relationship between differentials
in the homotopy fixed point spectral sequences for $\pi_*(B^{hG})$ and for
$\pi_*(\pics(B)^{hG})$. To this effect, the following comparison result is proved
in~\cite[Sec.~5]{mathew_stojanoska_picard}.
\begin{lemma}[{\cite[Comparison Tool
5.2.4]{mathew_stojanoska_picard}}] 
	Suppose that $A \to B$ is a faithful $G$-Galois extension. Whenever $2 \le r
	\le t-1$, there is an equality of differentials $d_r^{s,t}(\pics B) =
	d_r^{s,t-1}(B)$ in the Picard and additive descent spectral sequences. 
\end{lemma}

The above result arises from the fact that there is a canonical identification
of truncated spectra
\begin{equation}   \label{spectrumid} (\Sigma^{-1} \pics(R))_{[n, 2n-1]} \simeq R_{[n, 2n-1]}
\end{equation}
for any $n > 0$. The existence of the equivalence \eqref{spectrumid}
is a consequence of the Freudenthal suspension theorem. 
Of course, the main case of interest for these questions is when $s = t $. 

Our main technical advance in this paper is a strengthening of this comparison result when
certain primes are inverted in $\pi_0A$.
Motivated by the ``truncated'' logarithm in algebra, we obtain the following
expanded range of comparison of spectral sequences.  

\begin{theorem}\label{lem:diffimport}
	Suppose that $A \to B$ is a faithful $G$-Galois extension and $(p-1)!$ is
	invertible in $\pi_0 A$. Whenever $2 \le r \le (p-1)(t-1)$, there is an
	equality of differentials $d_r^{t,t}(\pics B) = d_r^{t,t-1}(B)$ in the
	Picard and additive descent spectral sequences. 
\end{theorem}

We also obtain a 
formula (in \cref{topunivformula}) for the first differential outside the above range in terms of power
operations. 

\subsection{Higher real $K$-theories}

The primary application of these methods in the present paper will be to the
higher real $K$-theory spectra of Hopkins--Miller. In particular, we will be
interested in the case where $B = E_n$, the $n$-th Morava $E$-theory, the
Landweber-exact ring spectrum with
\[
(E_n)_* \simeq \W(\F_{p^n})[\![u_1,\ldots,u_{n-1}]\!][u^{\pm 1}],
\]  
where $\W(\F_{p^n})$ denotes the Witt vectors over $\F_{p^n}$, described in more detail in~\Cref{wittvectorfree}. Here the $u_i$
have degree 0, while $u$ has degree $-2$. This is an $\einfty$-ring
spectrum~\cite{rezk_hm,goerss_hopkins} with an action through $\einfty$-ring
maps by a profinite group, $\G_n = \S_n \rtimes \Gal(\F_{p^n}/\F_p)$, the
(extended) Morava stabilizer group. Here $\S_n$ is the automorphism group of
the Honda formal group law over $\F_{p^n}$. 

By work of Devinatz and
Hopkins~\cite{dev_hop_04}, for any closed subgroup $G \subset \G_n$ there is an
associated $\einfty$-ring spectrum $E_n^{G}$, which acts like a homotopy fixed
point spectrum, and indeed, agrees with the usual construction when $G$ is
finite. For $G$ a maximal finite subgroup of $\G_n$ these sometimes are
called \emph{higher real $K$-theory spectra}; the nomenclature is inspired by the fact that when $n=1$ and $p=2$, the maximal finite subgroup of $\G_1 \simeq \Z_2^\times$ is $C_2 \simeq \{ \pm 1 \}$ and $E_1^{hC_2}$ is nothing other than 2-complete real $K$-theory.

Consider now the case of $G \subset \G_n$ a \emph{finite} subgroup. 
Then, the natural map $E_n^G = E_n^{hG} \to E_n$ is a $K(n)$-local $G$-Galois extension
by \cite[Thm.~5.4.4]{rognes_galois}
(see also~\Cref{sec:hmgalois} below),
so the Galois descent methods and the spectral sequence of~\eqref{eq:picss} are applicable. 
Using our general tools and the Gepner--Lawson spectral sequence, we prove the
following main result. 
\begin{theorem}\label{MainTheorem}
Let $p$ be an odd prime and $G \subset \G_{p-1}$ any finite subgroup. Then, the Picard group $\pic(E_{p-1}^{hG})$ is cyclic, generated by $\Sigma E_{p-1}^{hG}$. 
\end{theorem}

\newtheorem*{remark}{Remark}
\begin{rem} 
The theorem is also true when $n=1,p=2$, as can be shown using the same argument as given for $\pic(KO)$ in \cite{mathew_stojanoska_picard}.
\end{rem}
When $G \subset \G_{p-1}$ is a \emph{maximal} finite subgroup, the above result
is due in unpublished work to Hopkins. 
In this case, the result can be proved without the new technical tools we
develop in the paper. However, when (for example) $G = C_p \subset \G_{p-1}$, our formula for
the first differential is actually necessary to give an upper bound to the order of the Picard group of
$E_n^{hG}$. We note that Hopkins has informally observed that in most known examples, the
Picard groups of these types of constructions tend to be cyclic. 
Thus, our main result can be viewed as a further example of 
Hopkins's observation. 

\subsection{Connections with chromatic homotopy}
Let $K(n)$ be the $n$th Morava $K$-theory at an implicit prime $p$; it is a complex orientable spectrum with homotopy groups $\pi_* K(n) = \F_p[v_n^{\pm 1}]$ (where $|v_n| = 2(p^n-1)$), and its associated formal group law is the Honda formal group law.
It is an insight of
Hopkins~\cite{hms_94} that the Picard group $\pic_n$ of the $K(n)$-local category itself
is quite large and interesting to study.  

Here the calculation of $\pic_n$ essentially
breaks into two parts (which arise via the $K(n)$-local profinite Galois
extension $E_n^{\G_n} \to E_n$); there is an algebraic part, which can be
detected by cohomological methods, and an exotic part $\kappa_n$. The latter consists of the $K(n)$-local spectra such that the completed $E_n$-homology of $X$, $(E_n)^\vee_*X = \pi_*L_{K(n)}(E_n \wedge X)$ is isomorphic to $ (E_n)_*$ as Morava modules, i.e., as complete $(E_n)_*$-modules with a compatible action of $\G_n$ (see~\cite[Sec.~2]{ghmr_paper}). This latter part always vanishes when $p \gg n$, but can be highly non-trivial otherwise. 

Let us write $\pic_n^{\text{alg}}$ for the algebraic part\footnote{i.e. the Picard group of $\G_n$-equivariant $(E_n)_*$-modules.}, and $\kappa_n$ for the exotic part. 
In all \emph{known} cases the map $\pic_n \to \pic_n^{\text{alg}}$ is surjective (although it is not known if this is true in general), and there is a short exact sequence
\[
1 \to \kappa_n \to \pic_n \to \pic_n^{\text{alg}} \to 1. 
\] 
Moreover the generators of $\kappa_n$ and $\pic_n^{\text{alg}}$ are understood
well enough to determine the extension problem, so that the following results
completely determine $\pic_n$. However, the list of known cases is not very
long, and consists of the following. The height~1 case is due to
Hopkins--Mahowald--Sadofsky~\cite{hms_94}, and the height~2 case has many
contributions from a number of people: Hopkins for $p>3$ (unpublished, but
see~\cite{behrens_rev,lader}), Goerss--Henn--Mahowald--Rezk~\cite{ghmr_pic} for
$\kappa_2$ at $p=3$, and Karamanov~\cite{karamanov_pic} for the algebraic
Picard group $\pic_2^{\text{alg}}$ at $p=3$. The height~2 case for $p=2$ is the  subject of ongoing work of Beaudry--Bobkova--Goerss--Henn.
\begin{theorem}\label{thm:pickn}
	The Picard group of the $K(n)$-local category is given by:
	\[
\pic_1^{\text{alg}} = \begin{cases}
	\Z_2 \times \Z/2 \times \Z/2 & p=2 \\
	\Z_p \times \Z/2(p-1) & p>2. 
\end{cases}
\qquad 
\kappa_1 = \begin{cases}
	\Z/2& p=2 \\
	0 & p>2. 
\end{cases}
	\]
	\[
\pic_2^{\text{alg}} = \begin{cases}
	? & p=2 \\
	\Z_3^2 \times \Z/16 & p=3 \\
	\Z_p^2 \times \Z/2(p^2-1) & p>3. 
\end{cases}
\qquad 
\kappa_2 = \begin{cases}
	? & p=2 \\
	\Z/3 \times \Z/3 & p=3 \\
	0 & p>3.  
\end{cases}
	\]

\end{theorem}

It is known that $\Z_p$ always injects into $\pic_n$. However, apart from the vanishing of $\kappa_n$ for large primes, essentially nothing is known above height 2. However, the program of Henn and collaborators \cite{henn_finite, ghmr_paper,agnes_duality,bobkova_thesis} entails studying the $K(n)$-local sphere $S_{K(n)} = E_n^{\G_n}$ via finite resolutions by simpler homotopy fixed point spectra, namely $E_n^{hG}$ with $G$ a finite subgroup of $\G_n$. Once such a resolution of the $K(n)$-local sphere is understood, one can, in principle, build resolutions of any element of $\kappa_n$ by modifying the attaching maps in the resolution of the sphere \cite{ghmr_pic}. To proceed with this program, a first step is to understand the basic building blocks, i.e. $\pic(E_n^{hG})$.

Let $\kappa_{p-1}(G)$ denote the elements of $\pic(E_{p-1}^{hG})$ that are
detected in filtration $\ge 2$ in the Picard spectral sequence \eqref{eq:picss}
(in fact we will see that, by sparsity, such classes will only be in filtration
$\ge 2p-1$).
Our calculations show that $\kappa_{p-1}(G) \cong \Z/p$, whenever
$G$ contains an element of order $p$. 
One can check that base-change gives a well-defined morphism $\phi\colon\kappa_{p-1}
\to \kappa_{p-1}(G)$ which, in the cases of $n=1,\, p=2$ and $n=2, \, p=3$, and $G$ the maximal finite subgroup, is known to be split surjective, cf.~\cite[Sec.~8.5]{mathew_stojanoska_picard}, \cite[Ex.~5.1 and Thm.~5.5]{ghmr_pic}. This leads to the following question:
\begin{question}
	Is the map $\phi\colon\kappa_{p-1} \to \kappa_{p-1}(G) \simeq \Z/p$ non-zero for any $G$ containing an element of order $p$? 
\end{question}

 \subsection{Organization} 
The paper is arranged as follows.  In the next section we review the basic
properties of Morava $E$-theory, the classification of the maximal finite
subgroups of $\G_n$, and recall some information about the action of these
finite subgroups on $(E_{p-1})_*$. In \Cref{sec:lemmas} we prove some reduction results at general
height, in particular showing that one can reduce the proof of cyclicity of $\pic(E_n^{hG})$ to that of $\pic(E_n^{hG'})$, where $G'$ is the the maximal $p$-subgroup of $G$; we then perform the main computation of this paper in \Cref{sec:calcs}. This will rely on general techniques for comparing differentials in the additive and multiplicative algebraic and topological spectral sequences, which are described in the second part of the paper. More precisely, in \Cref{algTruncLogsec}, we analyze truncated logarithms to give comparison results and a universal formula for the differentials in the Picard spectral sequence, and in \Cref{sec:algPic} we use the techniques of \Cref{algTruncLogsec} to determine the algebraic Picard group $H^1(C_p, (E_{p-1})_0^\times)$.

\subsection*{Acknowledgments} 
We are grateful to Paul Goerss, Hans-Werner Henn, and Mike Hill for extremely helpful discussions.

The first and third authors thank the Max--Planck Institute for Mathematics in Bonn for its support and hospitality.   
The second author thanks the Hausdorff
Institute of Mathematics in Bonn for hospitality during the time which much of this
research was conducted. The second author was partially supported by the NSF
Graduate Research Fellowship under grant DGE-1144152.
The third author was partially supported by NSF grant DMS-1307390.

\part{Generalities and computations}

\section{Finite subgroups of $\G_n$ and the additive spectral sequence}\label{sec:finsubgroups}

In this section, we recall the necessary information about the action of finite
subgroups of $\G_n$ on $E_n$, particularly in the case $n=p-1$. We will,
in particular, recall the computation of the associated homotopy fixed point
spectral sequence for the homotopy of the higher real $K$-theories and their
variants for non-maximal subgroups. 
In the sequel, only the result of the computation  for
the group $C_p$ will be used.

\subsection{The Morava stabilizer group}\label{sec:stabgroup}
We start this section by reviewing more thoroughly the Morava stabilizer group $\S_n$, as well as the extended Morava stabilizer group $\G_n$.  We use 
\cite[Ch. III]{Frohlich68} as a general reference. 
Recall from the introduction that $\S_n$ is the automorphism group of the
height $n$ Honda formal group law $\Gamma_n$ over $\F_{p^n}$; this is the formal group law $F_{\Gamma_n}(X, Y) \in
\mathbb{F}_{p^n}[\![X, Y]\!]$ with $p$-series $[p]_{\Gamma_n}(X) = X^{p^n}$. 

\begin{definition} 
The group $\mathbb{S}_n$ of \emph{automorphisms} is given by those power series
$f(X) \in \mathbb{F}_{p^n}[\![X]\!]$ such that: 
\begin{enumerate}
\item $f(0) = 0$ (i.e., $f$ has no constant term).
\item $f'(0) \neq 0  \in \mathbb{F}_{p^n}$. 
\item $F_{\Gamma_n}( f(X), f(Y)) = f( F_{\Gamma_n}(X, Y))$.
\end{enumerate}
Observe that the group structure of $\mathbb{S}_n$ is given by composition of
power series. 
The group $\mathbb{S}_n$ has a topology induced from the $X$-adic topology, and
is a profinite group with respect to this topology. 
\end{definition} 

We have a homomorphism
\begin{equation}  \label{homomorphism} \mathbb{S}_n \to \mathbb{F}_{p^n}^{\times}, \quad f(X) \mapsto
f'(0).  \end{equation}
More invariantly, this homomorphism sends an automorphism $f$ of the formal
group $F$ to its action on the \emph{tangent space} of $F$ (dual to the
\emph{Lie algebra}). 
This homomorphism is surjective, and its kernel (which consists of
automorphisms that are congruent to the identity modulo $X^2$) is a pro-$p$-group. 

The group $\S_n$ can be described as the units in the maximal order $\cal{O}_n$
of the central division algebra $\mathbb{D}_n$ of Hasse invariant $1/n$ over $\Q_p$. More explicitly, we have that $\cal{O}_n$ is the non-commutative ring
\[
\cal{O}_n = \W(\F_{p^n})\langle S \rangle /(S^n=p,Sw = w^\phi S),
\]
where $w \in \W(\F_{p^n})$ and $\phi$ is a lift of the Frobenius to
$\W(\F_{p^n})$. Then $\cal{O}_n$ is the endomorphism ring of the formal group
$\Gamma_n$, and $\S_n = \cal{O}_n^\times$. 

Finally, because $\Gamma_n$ is already defined over $\F_p$, the Galois group $\Gal(\F_{p^n}/\F_p)$ acts on $\S_n$, and we define the extended Morava stabilizer group (which we will often just call the Morava stabilizer group) as the semidirect product $\S_n \rtimes \Gal(\F_{p^n}/\F_p)$, i.e., there is a split extension
\begin{equation}\label{eq:gnexactsequnce}
1 \to \S_n \to \G_n \to \Gal(\F_{p^n}/\F_p) \to 1. 
\end{equation}

\subsection{The finite subgroups of $\G_n$}
Before proceeding with the calculations, let us recall the classification of the maximal finite subgroups of $\G_n$, as well as their action on the coefficients of $E_n$. Bujard and Hewett~\cite{hewett_finite,bujard_finite} have determined the maximal finite subgroups of $\S_n \subset \G_n$. Specifically, let $p>2$ and $n=(p-1)p^{k-1}m$ for $m$ prime to $p$, and denote by $n_\alpha$ the quotient $n/\phi(p^\alpha)$, where $\phi$ is Euler's totient function. 
	There are exactly $k+1$ conjugacy classes of maximal finite subgroups of $\S_n$, represented by
\[
G_0 = C_{p^n-1} \quad \text{ and } \quad G_\alpha = C_{p^\alpha} \rtimes C_{(p^{n_\alpha}-1)(p-1)} \quad \text{ for } 1 \le \alpha \le k. 
\]
When $p-1$ does not divide $n$, the only class of maximal finite subgroups is that of $G_0$, so in particular, there are no finite subgroups of order divisible by $p$. 

\begin{rem}\label{rem:Sn}
Given the description of $\S_n$ in~\Cref{sec:stabgroup}, we note that an element of order dividing $p^n-1$ is always easy to describe: it is given by a Teichm\"uller lift $\eta$ of a primitive $(p^n-1)$st root of unity, $\eta \in \W(\F_{p^n})^\times \subset \S_n$. 
\end{rem}

Consider now the special case when $n = p-1$. For convenience, we drop the
subscripts after $\G, \S$. 
In this situation, there are two conjugacy classes of maximal finite subgroups of $\S$: 
\begin{enumerate}
\item 
One of order not divisible by $p $, namely the cyclic group $C_{p^n-1}$. 

\item  One of order divisible by $p$, namely the semi-direct product $F = C_p \rtimes C_{({p-1})^2}$. In the latter case, the action of $C_{({p-1})^2}$ on $C_p$ is given by a projection
\[
C_{{(p-1)}^2} \to C_{p-1} \simeq \operatorname{Aut}(C_p).
\] 
\end{enumerate} 

\begin{rem}
Using the description of $\S_n$ in~\Cref{sec:stabgroup}, we can give an
explicit equation of an order $p$ element when $p=3$. Let $\omega$ be a
primitive \nth{8} root of unity in $\W(\F_{9})$, then we can take $ -\frac 12 (1+ \omega S)$. In general, however, it is difficult or impossible to give a similar expression with rational coefficients for an element of order $p$ when $n>2$, see~\cite[Rem.~5.3]{gms}. 
\end{rem}

Finally, we recall from \eqref{eq:gnexactsequnce} that the extended Morava stabilizer group $\G$ sits in a split extension
\[ 1 \to \S \to \G \to \Gal(\F_{p^n}/\F_p) \to 1, \]
and therefore, any finite subgroup $G$ of $\G$ sits in a (not necessarily split) extension
\begin{align}\label{eq:Gdecomposition}
 1 \to G_{\S} \to G \to G_{\Gal} \to 1,
 \end{align}
where $G_{\S} = G\cap \S$, and $G_{\Gal} $ is a subgroup of $\Gal(\F_{p^n}/\F_p) \cong C_{n}$. 

In the case when $G_{\S}$ is a maximal abelian subgroup of $\S $, Bujard \cite[Thm.~4.13]{bujard_finite} has determined that $G_{\Gal}$ can be maximally large, i.e. it is always possible that $G_{\Gal} \cong \Gal(\F_{p^n}/\F_p)$. This is also known to be true when $G_\S$ is a maximal finite subgroup of $\S$, cf. \cite[Prop.~20]{henn_finite}. 

The proof of \Cref{MainTheorem} for $G$ will consist of a reduction first from $G$ to its subgroup $G_{\S}$, and then further to the $p$-Sylow subgroup of $G_{\S}$. 

\subsection{The action on $E_*$}

The first explicit description of the action of $\G$ on $E_*$ appears in unpublished work of Hopkins and Miller~\cite{hopkins_miller}. The general situation for $\G$ acting on $E_*$ is determined by Devinatz and Hopkins~\cite{dev_hop_action}.\footnote{Both of these sources treat general $p$ and $n$.} 
We focus on the action of $\S$. For our purposes, a nice summary of what we need can be found in the work of Nave~\cite[Sec.~2]{nave_smith_toda}. 

Write an element $g \in \S$ as $g = \sum_{j=0}^{n-1}a_j S^j$, where $a_j \in
\W(\F_{p^n})$ and $a_0$ is a unit. Recall that $u \in \pi_{-2} E_n$ is a
generator. The next result follows from \cite[Prop.~3.3 and Thm.~4.4]{dev_hop_action}.

\begin{theorem}
Let $\phi$ denote the lift of the Frobenius to $\W( \mathbb{F}_{p^n})$. Then if $g
= \sum_{j=0}^{n-1} a_j S^j$ as above, we
have: 
\begin{gather}
g_*(u) \equiv a_0u + a_{n-1}^\phi u u_1 + \cdots + a_1^{\phi^{n-1}} uu_{n-1}
\mod (p,\frak{m}^2). \\
g_*(uu_i) \equiv a_0^{\phi^i} uu_i + \cdots + a_{i-1}^{\phi}uu_1 \mod (p,\frak{m}^2).
\end{gather}
\end{theorem}

In the case of elements $g$ of finite order, we can be much more explicit. By the classification of finite subgroups of $\S$, such a $g$ either has order equal to $p$, or prime to $p$. 
\begin{enumerate}
\item 
If $g$ has order $p$, it is known~\cite[Lem.~2.2]{nave_smith_toda} that $a_0 \equiv 1 \mod (p)$ and $a_1$ is a unit. 
\item 
If $g$ has order prime to $p$, $g$ is a Teichm\"uller lift of $a \in \F_{p^n}^\times \subseteq \W^\times \subseteq \S$. Then it follows from~\cite[Eq.~8]{henn_finite} that the induced map of rings
\(
g_*\colon E_* \to E_*
\)
is given by
\begin{equation}\label{eq:actiona}
g_*(u) = au \quad \text{ and } g_*(u_i) = a^{p^i-1}u_i.
\end{equation}
\end{enumerate}



\subsection{The HFPSS}\label{sec:additivess}
Because the action of finite subgroups on $E_*$ is well understood in our case
of $n=p-1$ (which we implicitly assume in this subsection), so is the computation of the associated group cohomology. In fact, the differentials in the homotopy fixed point spectral sequence are also known. The calculations are due to Hopkins and Miller, although their work has not appeared in print. Published references include the thesis and subsequent paper of Nave~\cite{nave_thesis, nave_smith_toda}, or for a detailed computation at $n=2,p=3$, the work of Goerss--Henn--Mahowald and Goerss--Henn--Mahowald--Rezk~\cite{ghm_skye_paper,ghmr_paper}. 

We start by considering the prime-to-$p$ case. 
\begin{example} 
If $p$ does not divide the order of $G$, then the higher cohomology of $G$ with coefficients in $E_*$ vanishes, giving that
\[ (E^{hG})_* =  (E_*)^G.\]
Using \eqref{eq:actiona}, we can (at least in principle) determine these
invariants. In particular, we see that if $m$ is the order of $G_{\S}$, then
$u^m$ is $G_{\S}$-invariant and is a minimal periodicity element of $(E_*)^{G_{\S}}$. But this element is also fixed under the Galois quotient of $G$, so $E^{hG}$ has periodicity $2m=2|G_{\S}|$.
\end{example}

The most interesting case for us is when $G = C_p$. In this case, one has the
following calculation of the $E_2$-page modulo transfers. 
\begin{prop}[Hopkins--Miller]\label{prop:e2cals}
	There is an exact sequence
	\[
	\begin{split}
E_*\xr{\text{tr}} H^*(C_p,E_*) \to \F_{p^n}[\alpha,\beta,\delta^{\pm 1}]/(\alpha^2) \to 0,
	\end{split}
	\]
	of graded groups, where the $(s,t)$-bidegrees are $|\alpha| = (1,2n),|\beta| = (2,2pn),|\delta| = (0,2p)$. 
	
\end{prop}

	We note that because we have not taken Galois invariants anywhere, we see
	copies of $\mathbb{F}_{p^n}$ (rather than $\mathbb{F}_p$) in the above
	spectral sequence. The passage to $\F_p$ can be achieved by adding in the Galois group. For
	example, we also state the following result of Hopkins--Miller.  
	\begin{prop}[Hopkins--Miller]
When $G \subset \G$ is a maximal finite subgroup containing $p$-torsion, then there is an exact sequence
\[E_*\xr{\text{tr}} H^*(G,E_*) \to \F_{p}[\alpha,\beta,\Delta^{\pm 1}]/(\alpha^2),\]
with $|\Delta|=(0,2pn^2)$ and $\alpha, \beta$ as above.
\end{prop}

The above determines the $E_2$-page of the HFPSS for $\pi_*(E^{hG})$ in two important cases. 
The differentials in the HFPSS are also well known. For any finite subgroup $G \subset \G$, there are maps
\[
\Ext_{BP_*BP}(BP_*,BP_*) \to H^*(\G,E_*) \to H^*(G,E_*).
\]
If $G$ contains an element of order $p$, then, by~\cite{ravenel_arf}, the
composite sends $\alpha_1$ to $\alpha$ and $\beta_1$ to $\beta$, at least up to
a nonzero scalar. 

One can also show this directly using the functoriality of the Greek letter construction, as in~\cite[Prop.~7]{ghm_skye_paper}. For this, recall that there is a class $v_1 \in \Ext_{BP_*BP}^{0,2(p-1)}(BP_*,BP_*/(p))$ which defines a permanent cycle in the Adams--Novikov spectral sequence for the mod-$p$ Moore spectrum $V(0)$. Via the Greek letter construction, $v_1$ gives rise to an element $\alpha_1 \in \pi_{2p-3}S^0_{(p)} \simeq \Z/p$ detected by a class, which we also denote by $\alpha_1$, in $\Ext^{1,2p-2}_{BP_*BP}(BP_*,BP_*)$. Recalling that $v_1$ is sent to $u_1u^{1-p}$ under the orientation map $BP \to E$, one can check that the image of $v_1$ defines a non-trivial element in $H^0(G,E_*/(p))$. Using the short exact sequence
\[
0 \to E_* \to E_* \to E_*/(p) \to 0,
\]
one can see that this class has non-trivial image in $\delta^0(v_1) \in
H^1(G,E_{2p-2}) $ (which if $G$ is maximal is $\mathbb{F}_p$). The result follows since this latter group is generated by $\alpha$. One can make a similar argument for $\beta$, and since $\alpha_1$ and $\beta_1$ are permanent cycles, so are $\alpha$ and $\beta$. We also know via the calculations of Toda~\cite{toda67,toda68} that $\alpha_1\beta_1^p = 0$ and $\beta_1^{pn+1} = 0$ in $\pi_*S^0$. Sparsity in the homotopy fixed point spectral sequence, as well as the multiplicative structure, then allow the complete determination of the differentials.  

In the following we use the notation $\dequal$ to denote an equality that 
holds up to multiplication by a nonzero scalar.
\begin{lemma}\label{lem:differentials}
	In the spectral sequence for $\pi_*E^{hC_p}$, there are differentials
	\[
d_{2p-1}(\delta) \dequal \alpha \beta^{p-1}\delta^{1-n^2}   \quad \text{ and } \quad d_{2n^2+1}(\delta^{n^3} \alpha) \dequal \beta^{n^2+1}.
	\]
	If $G$ is a maximal finite subgroup, then in the spectral sequence for
	$\pi_*E^{hG}$, there are differentials  
	\[
d_{2p-1}(\Delta) \dequal \alpha \beta^{p-1} \quad \text{ and } \quad d_{2n^2+1}(\Delta^{n}\alpha) \dequal \beta^{n^2+1}. 
	\]
\end{lemma}

These differentials determine all others by multiplicativity. More precisely, in the $C_p$-HFPSS, the shorter differential $d_{2p-1}$ is non-zero on all $\beta^b \delta^d$ with $d\not\equiv 0 \mod p$, and classes of form $\alpha\beta^b \delta^d$ are hit whenever $b\geq p-1$ and $d\not\equiv -1 \mod p$. Now consider a monomial $\alpha\beta^b\delta^d$ that survives to the $E_{2p} = E_{2n^2+1}$-page; in particular, $d\equiv -1 \mod p$ or $b<p-1$. If $d\not\equiv -1 \mod p$ (so we would have that $b<p-1$), the potential elements for $d_{2n^2+1}(\alpha\beta^b \delta^d)$ (which are multiples of $\beta^{n^2+b+1}\delta^{d-n^3}$) have already been wiped out by $d_{2p-1}$.  When $d\equiv -1 \mod p$, the potential targets are still there, and \Cref{lem:differentials} implies that indeed $d_{2n^2+1}(\alpha\beta^b \delta^d) \dequal \beta^{n^2+b+1}\delta^{d-n^3}$.

As a specific example, note that $\alpha\delta$ is a permanent cycle, but $\alpha\delta^{-1}$ is not.

\begin{rem}\label{rem:periodicity}
This implies the following:
\begin{itemize}
    \item The homotopy groups $\pi_*E_{p-1}^{hC_p}$ are $2p^2$-periodic, with periodicity generator $\delta^p$, and, 
    \item The homotopy groups $\pi_*E_{p-1}^{hG}$ are $2p^2n^2$-periodic, with periodicity generator $\Delta^p$.
  \end{itemize}  
\end{rem}

We demonstrate the computation (modulo some elements on the 0-line) of $\pi_* E^{hC_p}$ for $n=2,p=3$  in~\Cref{fig:eo2}. 
\begin{figure}[!htbp]
\centering
\includegraphics{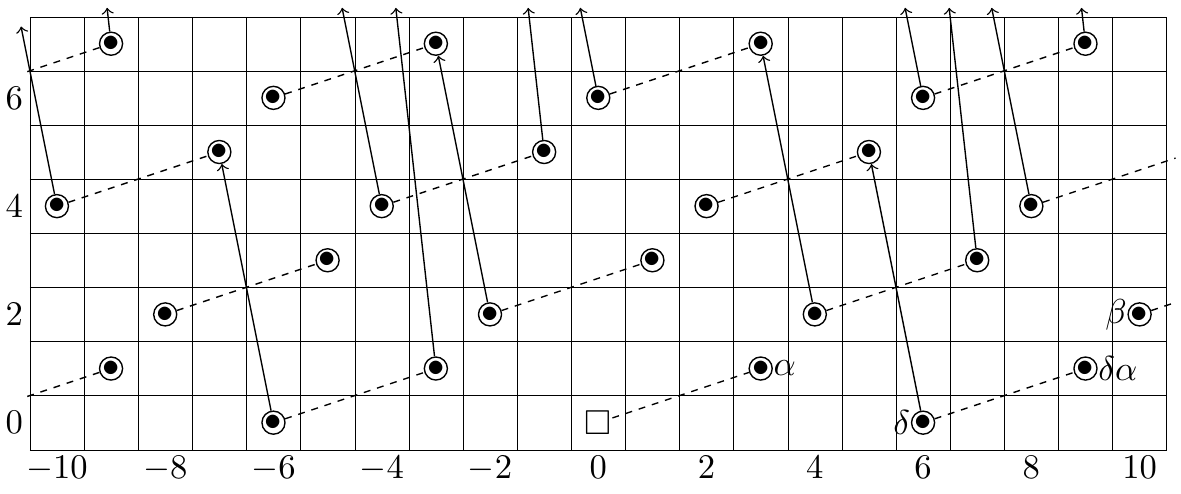}
\caption{The HFPSS for $\pi_*E^{hC_p}$ in the case $n=2,p=3$. Circled dots
refer to copies of $\F_9$, and the square represents a copy of $\W(\F_9)$.
The dotted lines represent multiplication by $\alpha$.}\label{fig:eo2}
\end{figure}


\section{General reduction steps}\label{sec:lemmas}

In this section, we will prove some general results about the action of finite
subgroups $G \subset \G_n$ on $E_n$ that will be needed for the computation
of the Picard group of $E_n^{hG}$ (i.e., the proof that it is cyclic). 
First, we will give a detailed proof of the
Devinatz--Hopkins--Rognes result (\Cref{galoisproperty}) that 
$E_n^{hG} \to E_n$ is a faithful $G$-Galois extension in the sense of Rognes
\cite{rognes_galois}, based on the affineness machinery of \cite{MMaffine}
(where the proof is sketched). 
In particular, we obtain the descent spectral sequence for Picard groups that
is our main tool. 

We then make two basic reductions for proving the cyclicity
of the Picard group, which show that it suffices to do the calculation in a
single case: when $G$ is a finite $p$-subgroup of $\S_n$. The analysis of the descent
spectral sequence in this case (in our special case of interest, $n=p-1$, $G$ will then be $C_p$) will be carried out in the following section, using the
technology developed in the second part of this paper. 

\emph{We emphasize that $n$ and $p$ are arbitrary in this section.}

\subsection{The Galois property of the Hopkins--Miller spectra}\label{sec:hmgalois}

In this subsection, we will show that  for any finite subgroup
$G \subset \G_n$, the extension $E_n^{hG} \to E_n$ is a faithful $G$-Galois
extension of ring spectra. 
This result is essentially due (at least in the $K(n)$-local case) to
Devinatz--Hopkins \cite{dev_hop_04} and Rognes \cite{rognes_galois}. 
We note also that, for $n = p-1$ (the case in which we are primarily
interested here), the first author in \cite{heard_tate} has given a
computational proof of the vanishing of the Tate construction, which is a key
component of the result. 

Using the results of \cite{MMaffine}, it will suffice to show that the stack
$(\spec \pi_0 E_n)/G $ is affine over the moduli stack $M_{FG}$ of formal
groups. This is easy to check at the closed point, and follows in general by a formal
argument that we describe here, involving ``free actions.''

We need to recall the following classical definition.

\begin{definition} 
Let $A$ be a commmutative ring with an action of a finite group $G$. Then $G$
is said to \emph{act freely} on $A$ if for every commutative ring $A'$ (it
suffices to take $A'$ to be a field), the $G$-action on the set $\hom(A, A')$
of ring-homomorphisms $A \to A'$ is free. 
Equivalently, for every prime ideal $\mathfrak{p} \subset A$, the subgroup
$G_{\mathfrak{p}} \subset G$ stabilizing $\mathfrak{p}$ acts faithfully on the
residue field $k(\mathfrak{p})$. In this case, the inclusion of rings $A^G \to
A$ induces a $G$-torsor $\spec A \to \spec A^G$, and $\spec A^G \simeq (\spec
A)/G$ in the category of stacks. 
\end{definition} 

We refer to \cite[Exp.~V]{SGA1} for a detailed treatment and to \cite[Sec.~4]{DAGXI} for a discussion with a view towards ring spectra. 
We will need the following important example of a free action in the future.

\begin{example}
\label{wittvectorfree}
We let $k$ be a perfect field of characteristic $p$. 
We recall (cf.~\cite[Ch.~II, Sec.~5--6]{Serreloc}) some properties of the functorial (in $k$) construction of the $p$-complete ring $\W(k)$ of $p$-typical Witt vectors, which is uniquely determined by the condition
that $\W(k)$ should be torsion-free and $\W(k)/p \simeq  k$. 
The ring $\W(k)$ has the following universal property for $p$-complete rings:
given any $p$-complete ring $A_0$, to give a map $\W(k) \to A_0$ is equivalent
to giving a map of $\mathbb{F}_p$-algebras $k \to A_0 /p$. 

We will need the following basic fact about Witt vectors. Let $k$ be a
perfect field of characteristic $p$ and let $G$ be any finite
subgroup of $\mathrm{Aut}(k)$. 
Then, the $G$-action on $\W(k)$ is free. 
In particular, the extension $\W(k^{G}) \to \W(k)$ is a $G$-Galois extension of
commutative rings (cf. \cite[Exp. I, Theorem 6.1]{SGA1}). 
\end{example}

We now specialize to the case of interest in this subsection. 

\begin{cons}
We consider the category consisting of pairs $(B, \Gamma)$ where $B$ is a
commutative ring and $\Gamma$ is a formal group over $B$. 
A morphism $(B, \Gamma) \to (B', \Gamma')$ consists of a morphism of
commutative rings $f\colon B \to B'$ and an isomorphism of formal groups $f_* \Gamma
\simeq \Gamma'$. 
This category 
is cofibered in groupoids over the category of commutative rings, and is
encoded in the \emph{moduli stack  of formal groups} $M_{FG}$: that is, it is (by
definition) the category of maps $\spec B \to M_{FG}$ from affine schemes to
$M_{FG}$.
\end{cons}

We will consider group actions in this category, or equivalently morphisms of algebraic stacks of the form $(\spec B)/G \to
M_{FG}$. In particular, giving a $G$-action on a pair $(B, \Gamma)$ yields for
each $\sigma \in B$ the data of an automorphism $\sigma\colon  B  \simeq B$ and an
isomorphism of formal groups $f_\sigma\colon  \sigma_* \Gamma \simeq \Gamma$ which satisfy a
homomorphism and cocycle condition, respectively. 
We will represent $\Gamma$ by a formal group law over $B$ given by a power
series $F(X, Y) \in B[\![X, Y]\!]$. In this case, 
$f_\sigma = f_\sigma(X)$ is a power series in $B[\![X]\!]$ such that
\[ F (f_\sigma (X), f_\sigma(Y)) = f_\sigma( \sigma_* F(X, Y)).   \]

We will need a general criterion to show that a map of the form $(\spec B)/G \to M_{FG}$ is
affine. 
For this, it will suffice to consider (for various $n$) the faithfully flat cover $M_{FG}^{(n)}
\to M_{FG}$ where 
$M_{FG}^{(n)}$ is the moduli stack of formal groups together with a coordinate
modulo degrees $\geq n+1$. 

\begin{prop} 
\label{freeafternstructure}
Suppose $G$ acts on the pair $(B, \Gamma)$ where $\Gamma$ is a 
a formal group over the commutative ring $B$ with the above notation. 
For any prime ideal $\mathfrak{p} \subset B$, we let 
$I_{\mathfrak{p}} \subset G$ be the subgroup of elements $\sigma \in G$ that
$\sigma( \mathfrak{p}) = \mathfrak{p}$ and $\sigma$ induces the identity
automorphism of $k( \mathfrak{p})$.

Fix $n$, and suppose that if $\mathfrak{p}$ is any prime ideal and $\sigma \in
I_{\mathfrak{p}}$, then at least one of the first $n$ coefficients of $f_\sigma(X) - X  \in B[\![X]\!]$ 
has nonzero image in $k( \mathfrak{p})$.
Then the $G$-action on the scheme $\spec B \times_{M_{FG}} M_{FG}^{(n)}$ is
free, so that $(\spec B/G) \times_{M_{FG}} M_{FG}^{(n)}$ is representable by an
affine scheme.
\end{prop}
\begin{proof} 
Let $B'$ be the $B$-algebra which classifies coordinates modulo degrees $\geq
n+1$ on the formal
group $\Gamma$. Abstractly, $B' \simeq B[x_1^{\pm 1}, x_2, \dots, x_{n}]$. 
The $G$-action on $(B, \Gamma)$ induces a $G$-action on $(B', \Gamma')$, where
$\Gamma'$ is the extension of scalars of $\Gamma$  to $B'$. 

We need to show that $G$ acts freely on $\hom(B', k)$ for any field $k$. 
Let $\sigma \in G \setminus \left\{1\right\}$ fix an element of $\hom(B', k)$ given by a map $\phi\colon  B' \to k$. 
We get that $\sigma$ fixes the restriction $\phi|_B\colon  B \to k$, which means that 
the prime ideal $\mathfrak{p} = \mathrm{ker}( \phi|_B)$ is fixed by $\sigma$,
and that $\sigma$ acts trivially on the residue field $k( \mathfrak{p})$, so
that $\sigma \in I_{\mathfrak{p}}$. 
Note that the fiber above $\phi|_B$ of the map (which is preserved by $\sigma$)
\begin{equation} \label{maptotakefiber} \hom(B', k) \to \hom(B, k) \end{equation} 
is given by the set of coordinates to degree $n$ on the formal group $(\phi|_B)_* (\Gamma)$,	
and the action of $\sigma$ on this set of coordinates is given by composing
with the image of $f_{\sigma}$. 
However, since the image of $f_\sigma \in k[\![X]\!]$ is not the identity, 
it acts freely on the above set of coordinates up to degree $n$. In particular, the 
action of $\sigma$ on the fiber of \eqref{maptotakefiber} above $\phi|_B$ has
no fixed points, which is a contradiction and 	which completes the proof. 
\end{proof} 

\begin{lemma} 
\label{keyaffinenesslemma}
Suppose $(R, \mathfrak{m})$ is a complete local ring with perfect residue field
$k$ of characteristic $p > 0$, and let $\Gamma$ be a formal group over
$R$. 
Let a finite group $G$ act on the pair $(R, \Gamma)$. Let $S \subset G$ be the
subgroup that acts trivially on the residue field. Suppose that every
element $\sigma \in S \setminus \left\{1\right\}$ induces a nontrivial
automorphism of the reduction of $\Gamma$ modulo $\mathfrak{m}$. 
Then the morphism $(\spec R)/G \to M_{FG}$ is affine.
\end{lemma} 
\begin{proof} 
Fix a prime ideal $\mathfrak{p} \subset R$ and $\sigma \in I_{\mathfrak{p}}$. The first claim is that 
$\sigma \in S$. 
In fact, if $C$ is any ring, we claim that any $\sigma \notin S$ acts without
fixed points on $\hom(R, C)$. 
To see this, we observe that we have a $G$-equivariant map $\W(k) \to R$ using
the universal property of the Witt vectors.  
In particular, we have a map
\[ \hom(R, C) \to \hom(\W(k), C),  \]
compatible with the $\sigma$-action on both sides, and it suffices to show that
$\sigma$ acts freely on the target; this follows from \Cref{wittvectorfree},
since $G/S \subset \mathrm{Aut}(k)$
acts freely on $\W(k)$. 

Therefore, it suffices to consider the case when $G$ itself
acts trivially on the residue field. 
Choose $n$ large enough such that each power series $f_{\sigma}(X) \in R[\![X]\!]$ for $\sigma \neq 1$ has a
coefficient in degrees $\leq n $ different modulo $\mathfrak{m}$ from $X$. 
The claim is that $(\spec R)/G \times_{M_{FG}} M_{FG}^{(n)}$ is representable
by an affine scheme. This now follows from \Cref{freeafternstructure} because
each $f_{\sigma}(X) - X$ for $\sigma \neq 1$ has a coefficient which is a
\emph{unit}
in some degree $\leq n$. 
\end{proof}

\begin{prop} 
\label{galoisproperty}
For any finite subgroup $G \subset \G_n$, the
natural map $E_n^{hG} \to E_n$ exhibits the target as a faithful
$G$-Galois extension of the source. \end{prop} 
\begin{proof} 
By \cite[Thm.~5.8]{MMaffine}, it suffices to show that $(\spec \pi_0 E_n)/G \to
M_{FG}$ is an affine morphism of stacks. 
However, this follows 
from \Cref{keyaffinenesslemma} (with $S = G_\S = G \cap \S $)
because $\S_n$ is defined as a subgroup of automorphisms
of the Honda formal group over $\mathbb{F}_{p^n}$. 
The above reasoning shows that there exists $m$ such that 
$(\spec \pi_0 E_n)/G \times_{ 
M_{FG}} M_{FG}^{(m)}$ is affine. 
\end{proof} 

It follows in particular that there is a descent spectral sequence for
computing $\pic(E_n^{hG})$ (for any finite subgroup $ G \subset \G_n$) based on the equivalence of spectra
\[ \pics(E_n^{hG}) \simeq \tau_{\geq 0}  \pics(E_n)^{hG}.  \]
We know that $\pic(E_n)$ is cyclic (by regularity, cf. \cite{baker_richter_invertible}), and we want to
prove that $\pic(E_n^{hG})$ is also cyclic. 	
In the rest of this section, we will reduce the analysis of this spectral
sequence to the case where $G$ is a $p$-subgroup of $\S_n$.

\begin{rem} 
As a consequence, we can prove that the two possible notions of
Picard group that one might be interested in here are equivalent. Recall that
if $R$ is a $K(n)$-local $\einfty$-ring, 
one can consider two different  Picard groups (cf. \cite{hms_94}): 
\begin{enumerate}
\item The (usual) Picard group $\pic(R)$ of all $R$-modules.  
\item The Picard group $\pic_{K(n)}(R)$ of the $\infty$-category $L_{K(n)} \Mod(R)$ of
$K(n)$-local $R$-modules (with the $K(n)$-localized smash product).
\end{enumerate}
The first group $\pic(R)$ is always a subgroup of $\pic_{K(n)}(R)$, because any invertible
$R$-module is automatically perfect, hence $K(n)$-local itself and is easily
seen to be $K(n)$-locally invertible too. 
Given an element of $\pic_{K(n)}(R)$, represented by a $K(n)$-local $R$-module
$M$,
it belongs to the first group if and only if $M$ is perfect as an $R$-module.  

We claim that if $R = E_n^{hG}$ for any finite $G  \subset \G_n$, the two notions
coincide. 
Equivalently, if $M$ is a $K(n)$-local $R$-module which is 
$K(n)$-locally invertible, then $M$ is invertible as an $R$-module to begin with. 

To see this, we consider the faithful $G$-Galois extension $R \to E_n$. 
Then $E_n$ is a perfect $R$-module by \cite[Prop. 6.2.1]{rognes_galois}. 
Observe that the $E_n$-module $E_n \otimes_R M$ 
is still $K(n)$-local as a result, and it is $K(n)$-locally invertible. 
This implies that $E_n \otimes_R M$ belongs to $\pic_{K(n)}(E_n)$. However,
$\pic_{K(n)}(E_n) = \pic(E_n) = \mathbb{Z}/2$ (by \cite{baker_richter_invertible}); in particular,
$E_n \otimes_R M$ is a perfect $E_n$-module, and by descent (cf.~\cite[Prop. 3.27]{galois}) $M$
is a perfect $R$-module. 
\end{rem} 

\subsection{Elimination of the Galois piece}

In this subsection, we make a basic reduction that enables us to eliminate the
Galois piece. In other words, we reduce to proving cyclicity of the Picard
group of $E_n^{hG}$ when $G \subset \S_n$.

We will need a few basic properties of $\W(k)$. Most of them can be encapsulated
appropriately in the notion of ``formal \'etaleness,''
but we will spell out some details for the convenience of the reader. 

Let $f\colon  A_0 \to A_0'$ be  a morphism of $p$-local commutative rings. 
We say that $f$ is a \emph{uniform $\mathcal{F}_p$-isomorphism} if there
exists $N > 0$ such that: 
\begin{enumerate}
\item $\mathrm{ker}(f)^N = 0$. 
\item Given $x \in A_0'$, $x^{p^N} \in \mathrm{im}(f)$.
\end{enumerate}

\begin{lemma}
Let $f\colon  A_0 \to A_0'$ be a uniform $\mathcal{F}_p$-isomorphism of $p$-complete rings.
Then given any perfect field $k$ of characteristic $p$, any map of commutative
rings $\W(k) \to A_0'$ lifts uniquely to a map $\W(k) \to A_0$.
\end{lemma} 
\begin{proof} 
It is a straightforward calculation (cf.~\cite[Prop.~3.28]{MNN1}) that $f
\otimes_{\mathbb{Z}_p} \mathbb{F}_p$ is also a uniform $\mathcal{F}_p$-isomorphism, so
that we may assume that $A_0$ and $A_0'$ are $\mathbb{F}_p$-algebras by the
universal property of the Witt vectors. 
It suffices to show that $\hom_{}(k, A_0) \simeq
\hom_{}(k, A_0')$. 

Recall that an $\mathbb{F}_p$-algebra $C_0$ is said to be \emph{perfect} if 
the Frobenius $F\colon  C_0 \to C_0$ is an isomorphism. The inclusion of perfect
$\mathbb{F}_p$-algebras into all $\mathbb{F}_p$-algebras has a right adjoint,
called the \emph{inverse limit perfection,}
which sends an $\mathbb{F}_p$-algebra $C_0$ to the inverse limit of the Frobenius system 
$\dots \to C_0 \stackrel{F}{\to} C_0$.  
Since $k$ is perfect, and since the uniform $\mathcal{F}_p$-isomorphism condition
guarantees that $A_0 \to A_0'$ induces an isomorphism on inverse
limit perfections, an
adjunction argument now gives the desired claim. 
\end{proof}

\begin{lemma} 
\label{liftalongfixedhomotopyfixed}
Let $\tilde{G}$ be a finite group acting on a $p$-complete $\einfty$-ring $R$ in such a way that
$R^{h\tilde{G}} \to R$ is a faithful $\tilde{G}$-Galois extension. 
Let $k$ be a perfect field of characteristic $p$. Then any map of
commutative rings $\W(k) \to
\pi_0(R)^{\tilde{G}}$ lifts uniquely to a map $\W(k) \to \pi_0(R^{h\tilde{G}})$. 
\end{lemma} 
\begin{proof} 
By the previous lemma, it suffices to show that the reduction map $\pi_0
(R^{h\tilde{G}}) \to \pi_0(R)^{\tilde{G}}$ is a uniform $\mathcal{F}_p$-isomorphism. 
This follows as in \cite[Sec.~4]{Mathick} from the fact that the 
HFPSS for the $\tilde{G}$-action on $R$ has a horizontal vanishing line at a finite
stage by the theory of ``descendability'' (cf. \cite[Sec.~3--4]{galois}) and
the fact that everything in positive filtration at $E_2$ is annihilated by
the $p$-part of $|\tilde{G}|$. 
\end{proof}

\begin{prop}\label{prop:galoisreduction}
Let $G$ be a finite group acting on an even-periodic $\einfty$-ring $R$.
Suppose that:
\begin{enumerate}
\item  
$\pi_0 R$ is a complete local ring with perfect residue field
$k$ of characteristic $p > 0$. 
Let $\tilde{G} \subseteq G$ be the subgroup that acts trivially on $k$.
\item
 The morphism of $\einfty$-rings $R^{hG} \to R$ is a faithful $G$-Galois extension.
 \item 
The Picard group of the $\einfty$-ring $R^{h\tilde{G}}$ is
cyclic and generated by the suspension $\Sigma R^{h\tilde{G}}$. 
\end{enumerate}
Then the Picard group of $R^{hG}$ is
cyclic and generated by the suspension $\Sigma R^{hG}$.
\end{prop} 
\begin{proof} 
Let $\Gamma = G/\tilde{G}$ be the quotient group, so that by the Galois
correspondence, the morphism of $\einfty$-rings $A = R^{hG} \to B =  R^{h\tilde{G}}$ is a
faithful $\Gamma$-Galois extension. 
We claim first that it is actually a Galois extension on homotopy groups, i.e., $A
\to B$ is \'etale, in the sense that $\pi_0(A) \to \pi_0(B)$ is \'etale and the natural map $\pi_*A \otimes_{\pi_0A}\pi_0B \to \pi_*B$ is an isomorphism; this will enable us to make the desired conclusion. 
We refer to \cite{BakerRichteralg} for the relationship between Galois
extensions on homotopy and Galois extensions in the sense of Rognes
\cite{rognes_galois}, and to \cite[Sec.~4]{DAGXI} as a convenient reference for some of the facts
we need. 

In particular, by \cite[Cor.~4.15]{DAGXI}, we need to show that $\Gamma$ acts
freely on $\pi_0(B)$. We will do this by
functoriality. 
By the universal property of the ($p$-typical) Witt vectors, 
we obtain a $G$-equivariant map
\[ \W(k) \to \pi_0 (R),  \]
where $G$ acts on $\W(k)$ via the induced action on $k$. 
Of course, the $G$-action on $\W(k)$ factors through $\Gamma = G/\tilde{G}$, and we claim that
this lifts to a $\Gamma$-equivariant map
\[ \W(k) \to \pi_0(B) = \pi_0(R^{h\tilde{G}}).  \]
The equivariance alone gives the $\Gamma$-equivariant lift $\W(k) \to
\pi_0(R)^{\tilde{G}}$, which of course receives a map from $\pi_0(R^{h\tilde{G}})$. 
Using \Cref{liftalongfixedhomotopyfixed}, 
we can upgrade it to the claimed $\Gamma$-equivariant map $\W(k) \to \pi_0(B)$.

The $\Gamma$-action on $W(k)$ is free (\Cref{wittvectorfree}). 
It follows that the $\Gamma$-action on $\pi_0(B)$ is free.
Applying \cite[Cor.~4.15]{DAGXI}, we conclude that $A \to B$ is an algebraic
$\Gamma$-Galois extension in that $\pi_0(A) \to \pi_0(B)$ is one, and $\pi_*(A)
\otimes_{\pi_0(A)} \pi_0(B) \to \pi_*(B)$ is an isomorphism. 

Finally, we prove the result about the Picard group. 
Let $M$ be an invertible $A$-module. 
By assumption, $M \otimes_A B$ is a suspension of a free $B$-module of rank
one; by desuspending $M$ appropriately, we may assume that $M \otimes_A B
\simeq B$. We would like to see that $M \simeq A$. 
Since $B$ is faithfully flat over $A$, 
we can conclude that  $\pi_0(M)$ is an invertible $\pi_0(A)$-module
and that the map $\pi_0(M) \otimes_{\pi_0(A)} \pi_*(A) \to \pi_*(M)$ is an
isomorphism (because this happens after base-change to $B$). 
In particular, $M$ simply arises from an invertible module over $\pi_0(A)$. 

Therefore, it suffices to show that the Picard group of the (ordinary)
commutative ring $\pi_0(A)$ is trivial. This will follow if we can show that
$\pi_0(A)$ is a (possibly non-noetherian) local ring.  Since $\pi_0(A) \to \pi_0(B)$
is finite \'etale, it suffices to show that $\pi_0(B)$ is a local ring, i.e.,
that its spectrum has a unique closed point. However, $\pi_0(B) \to \pi_0(R)$
is an $\mathcal{F}_p$-isomorphism as above, so that it induces an isomorphism
on Zariski spectra. Therefore, the fact that $\pi_0(R)$ is a local ring forces
$\pi_0(B)$ to be one, and completes the proof. 
\end{proof} 

In our example of interest, $R$ will be $E_n$, $G$ will be a finite subgroup of the extended Morava stabilizer group, and $\tilde{G} = G_{\S}$, as in \eqref{eq:Gdecomposition}.

\subsection{Reduction to the $p$-Sylow}

Here, we will use basic facts about the structure of $\mathbb{S}_n$, for which we refer the reader to~\Cref{sec:stabgroup}. 

\begin{prop}\label{prop:sylowreduction}
Let $G$ be a finite subgroup of $\mathbb{S}_n$ at
some height $n$ and at the implicit prime $p$. 
Let $G_p \subseteq G$ be a $p$-Sylow subgroup.
Suppose the Picard group of the $\einfty$-ring $E_n^{hG_p}$ is cyclic,
generated by the suspension. Then the Picard group of $E_n^{hG}$ is also
cyclic, generated by the suspension.
\end{prop} 

\begin{proof} 
Observe first that $\pic(E_n^{hG}) = \pi_0 (\pics(E_n)^{hG})$ is a torsion
group (cf. \cite[Sec.~5.3]{mathew_stojanoska_picard}), so it suffices to show
that the suspension generates the $q$-part $\pic(E_n^{hG})_q$ of
$\pic(E_n^{hG})$ for any prime
number $q$. 
We consider two cases: 
\begin{enumerate}
\item $q = p$. In this case, the map $$\phi_p\colon  \pic(E_n^{hG})_p  \to
\pic(E_n^{hG_p})_p$$
(induced by base-change along $E_n^{hG} \to E_n^{hG_p}$)
is injective thanks to the transfer map
\[ \pic(E_n^{hG_p})_p \simeq \pi_0 \pics( E_n)^{hG_p} \to  \pic(E_n^{hG})_p
\simeq \pi_0  \pics(E_n)^{hG}. \]
Namely, recall that if $X$ is a $p$-local spectrum with a $G$-action, the composite map
$X^{hG} \to X^{hG_p} \stackrel{\mathrm{tr}}{\to}X^{hG}$ is an equivalence. 
Since the map $\phi_p$ carries the suspension of the unit to itself, the fact that
this class generates the target implies that it must generate the source. 
\item $q \neq p$. Let $G_q \subseteq G$ be a $q$-Sylow subgroup. 
As above, we observe that the map 
$\phi_q$ 
defined via
$$\phi_q\colon  \pic(E_n^{hG})_q  \to
\pic(E_n^{hG_q})_q$$
is injective. 
So, it suffices to consider the case where $G$ is itself a $q$-group. 

We consider the HFPSS converging to $\pic( E_n^{hG}) = \pi_0
\pics(E_n)^{hG}$. 
In this case, since $G$ is a $q$-group, we observe that the only contributions come from $H^0( G,
\mathbb{Z}/2)$
(which is $\Z/2$, generated by the suspension) and $H^1(G, (E_n)_0^{\times}) \simeq H^1(G,
(\mathbb{F}_{p^n})^{\times})$ via the Teichm\"uller lift. In other words,
we need to show (in order to prove cyclicity of the Picard group in this
case) that 
the group 
$H^1(G,
(\mathbb{F}_{p^n})^{\times}) = \hom(G, \mathbb{F}_{p^n}^{\times})$ is generated 
by the class $\psi$ corresponding to the $G$-action on $\pi_{-2}( E_n)
\otimes_{(E_n)_0} \mathbb{F}_{p^n}$. (This will give the order of the Picard group as well as solve the extension problem.)
However, we know (by general facts about even-periodic ring spectra) that this is simply the action on 
the Lie algebra, i.e., the homomorphism of \eqref{homomorphism}. 

But as we remarked before, the
kernel of \eqref{homomorphism} is pro-$p$, so its restriction to $G$ must be injective. Hence $G$ injects into $\mathbb{F}_{p^n}^{\times}$, which
forces $H^1(G, \mathbb{F}_{p^n}^{\times})$ to be generated by the class of the
injection. 
\end{enumerate}
\end{proof}

\section{Cyclicity of the Picard group}\label{sec:calcs}

We now fix an odd prime $p$, and work at height $n=p-1$; we will omit subscripts from our notation. In this section, we will use the homotopy fixed point spectral sequence to compute
$\pic(E^{hG})$ for all finite subgroups $G \subset \G$ , referring to future sections of the
document for some technical results required. We will use the results of the previous section to reduce to the case where $G = C_p$ is the $p$-Sylow subgroup of a maximal finite subgroup of $\S$. We also sketch a direct proof when $G$ is a maximal finite subgroup of $\G$ with $p$-torsion, to illustrate to the reader that taking Galois invariants simplifies the calculation somewhat. 

\subsection{The general case}\label{sec:cpcomputation}
Our goal in this subsection is to prove the main result of this paper, namely the following:
\begin{theorem}
	Let $G \subset \G$ be a finite subgroup. Then $\pic(E^{hG})$ is cyclic, generated by $\Sigma E^{hG}$. 
\end{theorem}
By combining~\Cref{prop:galoisreduction,prop:sylowreduction}, it is enough to
prove the result when $G = C_p$. To handle this case, we start by recalling from~\Cref{prop:e2cals} that, modulo the image of the transfer,
\begin{equation}\label{eq:cpe2}
H^*(C_p,E_*) \simeq \F_{p^n}[\alpha,\beta,\delta^{\pm 1}]/(\alpha^2)
\end{equation}
where $|\alpha| = (1, 2n), |\beta| = (2, 2pn),$ and $|\delta| = (0,2p)$.
Recall also that the homotopy groups $\pi_*(E^{hC_p})$ are $2p^2$-periodic; this sets $\Z/(2p^2)$ as a lower bound on the order of the Picard group. We will use the Picard spectral sequence to determine an upper bound.
Since we will be comparing the Picard spectral sequence with the additive one
(for computing $\pi_* E^{hC_p}$),
we will write $d_{k, \times}$ for the Picard differentials and $d_k$ for the
additive ones. 

Our first task is to work out the $E_2$-term $H^s(C_p,\pi_t\pics (E))$, and we are interested in the abutment for $t-s=0$. In the range where $t\ge2$, this follows immediately from~\eqref{eq:cpe2}. Namely, we have
\[ H^s(C_p, \pi_t \pics (E) ) = \left( \F_{p^n}[\alpha,\beta,\delta^{\pm 1}]/(\alpha^2) \right)^{s,t-1}, \qquad \text{when } t\geq 2, \]
since in this case $\pi_t \pics(E) = \pi_{t-1} E$, equivariantly. (The superscript notation on the right hand side designates the bidegree.)

Let us single out the classes in the $t-s=0$ column, i.e., those with potential contribution to the Picard group of $E^{hC_p}$, that come from the additive spectral sequence (so, $s> 1$).
They are
\[
e_c = \alpha \beta^{pc-1} \delta^{n-1-c(pn-1)}, 
\]
indexed by integers $c \ge 1$, and have cohomological degree $s(e_c) = 2pc-1$. 
\begin{lemma}\label{lem:e_c}
  The classes $e_c$ for $c \ne 1$ do not survive the Picard spectral sequence. 
\end{lemma}
\begin{proof}
In the \emph{additive} spectral sequence,  by \Cref{lem:differentials}, the classes $e_c$ are never the source of a $d_{2p-1}$ differential, but can be the target. In particular, there are additive differentials
\[d_{2p-1}(\beta^a \delta^b) \dequal  b \alpha \beta^{a+p-1} \delta^{b-n^2}.  \]
For this image to be (a multiple of) $e_c$, we must have $b = n^2+n-1 - c(pn-1) \equiv c -1 \mod p$. 
\begin{itemize}
\item If $c\not \equiv 1 \mod p$, then $e_c$ is the target of $d_{2p-1}$ on an element with $t=2p(c-1)$. Since such $c$ is $\geq 2$, we have that  $t\geq 2p>2p-1$, so that~\Cref{lem:diffimport} implies that this differential can be imported in the Picard spectral sequence. Hence $e_c$ itself is killed. 
\item If $c\equiv 1 \mod p$, this differential is zero (as then $b\equiv 0 \mod p$), so $e_c$ survives the $d_{2p-1}$. 
\end{itemize}

The next and only possible additive differential is a $d_{2n^2+1}$. We have 
\[d_{2n^2+1}(\alpha \delta^{n^3}) \dequal \beta^{n^2+1}, \]
which determines the rest. Namely, we have, for all $a,b\in \Z$, $a\geq 0$,
\[d_{2n^2+1}(\alpha\beta^a \delta^{bp-1}) \dequal \beta^{a+n^2+1} \delta^{pb-n^3-1}. \]
When $c\equiv 1 \mod (p)$, we have that $n-1-c(pn-1) \equiv -1 \mod p$, hence
\[ d_{2n^2+1}(e_c) \dequal \beta^{pc+n^2} \delta^{c(1-pn) + n-n^3-1}. \]
To import this differential to the Picard spectral sequence, we need to have that 
\[ 2n^2+1 \leq t(e_c) = 2pc-2. \]
This holds whenever 
\[ p-2 +\frac{5}{2p} \leq c \quad \Leftrightarrow \quad c\geq p-1. \]
Since $c \equiv 1 \mod (p)$, we see we can import all differentials except that involving $e_1$, and the result follows.
\end{proof}
\begin{rem}
	This lemma says nothing about the fate of $e_1$; it does not survive in the additive spectral sequence, but we shall see later that it must survive in the multiplicative case. 
\end{rem}
We can now prove our main result. 
\begin{proof}[Proof of Theorem 4.1]
We have seen that in the range $t \ge 2$ of the homotopy fixed point spectral sequence for $\pi_{t-s}\pics(E^{hC_p})$, we only have a single possible contribution to the $0$-stem, namely a group of order at most $p^n$, arising from the class $e_1$. This class has bidegree $(2p-1,2p-1)$, and it follows from \Cref{topunivformula} that there is a differential
\begin{equation}\label{eq:diff}
d_{2n^2+1,\times}(e_1) \dequal d_{2n^2+1}(e_1) + \zeta \beta \cal{P}^{n/2}(e_1),
\end{equation}
 for $\zeta \in \F^\times_p$. Let $f_t = \beta^{n^2+n+1}\delta^{-n^2(n+1)}$ so that we have
 (in the additive spectral sequence)
 \[
d_{2n^2+1}(e_1) \dequal f_t. 
 \]

 While it is possible to work out exactly what $\beta \cal{P}^{n/2}(e_1)$ is,
 we will not need to be so precise. Instead we claim that the group
 $\mathbb{F}_{p^n}e_1$ on the $E_2$-page  can contribute at most an element of order $p$ to Picard group. 
 
Choose $\xi' \in \mathbb{F}_{p^n}$ such that $\beta \mathcal{P}^{n/2}(e_1) =
\xi' f_t$. 
  Then we have 
 \[
d_{2n^2+1}(e_1) = \xi f_t \quad \text{and} \quad \beta\cal{P}^{n/2}(e_1) = \xi'f_t,
 \]
 for $\xi,\xi' \in \F_{p^n}$ and $\xi \neq 0$. It follows from \eqref{eq:diff} that, for any $a \in \F_{p^n}$, we have
 \[
\begin{split}
	d_{2n^2+1,\times}(ae_1) &= d_{2n^2+1}(ae_1) + \zeta \beta \cal{P}^{n/2}(a e_1) \\
	& = a\xi f_t + a^p \zeta \beta \cal{P}^{n/2}(e_1) \\
	&= a \xi f_t + a^p \zeta \xi' f_t \\
	&=(a\xi + a^{p}\zeta \xi') f_t. 
\end{split}
 \]
Since $\xi \ne 0$, there are at most $p$ choices of $a$ that make the above vanish, and so we do indeed have an upper bound of $\Z/p$ for the contribution of $e_1$ in $\pic(E^{hC_p})$. 

Finally, we know from~\Cref{prop:algebraicterm} that $H^1(C_p,\pi_1\pics(E))
\simeq \Z/p$, while as usual it follows from~\cite{baker_richter_invertible} that the invariants $H^0(C_p,\pi_0\pics(E))$ are $\Z/2$. Thus we have the following in the 0-stem of the Picard spectral sequence:

\begin{itemize}
  \item a group of order at most $2$ in $t = s = 0$, 
  \item a group of order at most $p$ in $t = s =1 $, and
  \item a group of order at most $p$ in $t = s =2p-1$.
\end{itemize}
This gives an upper bound of $2p^2$ for the order of the Picard group, but as noted previously this is also a lower bound, and the result follows. 
A picture of the Picard spectral sequence for $n=2,p=3$ and $G = C_3$ is shown in~\Cref{fig:eo2pic}.
\begin{figure}[!thbp]
\centering
\includegraphics{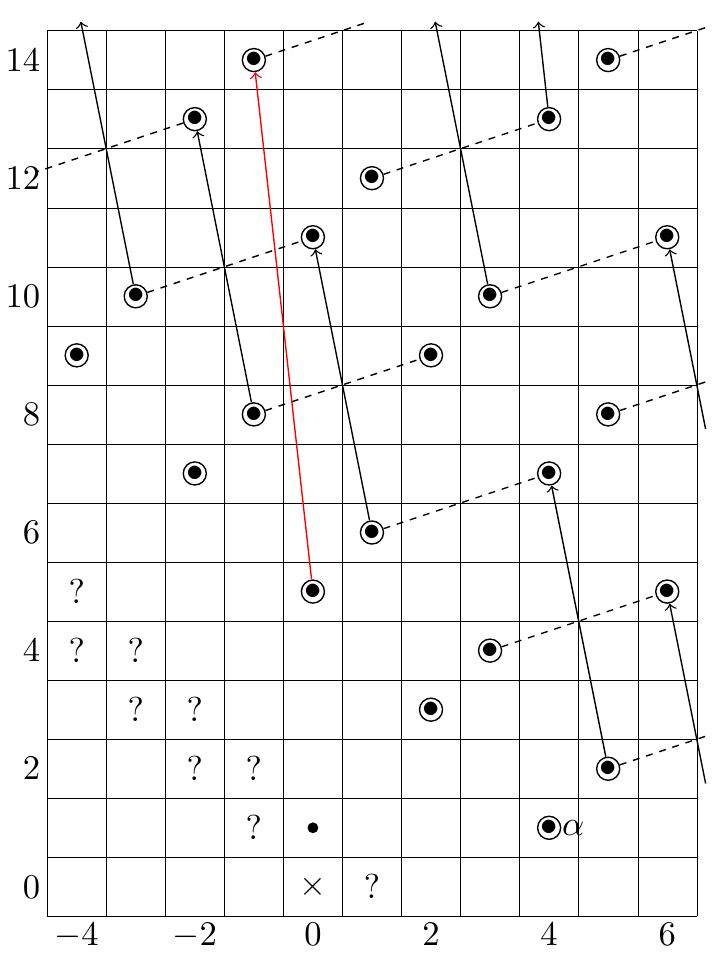}
\caption{The HFPSS for $\pi_*\pics E_{p-1}^{hC_p}$ in the case $n=2,p=3$. The differential shown in red has kernel $\Z/3$.}\vspace{-4mm} 
\caption*{  ($\bullet$ denotes $\Z/3$, $\cbull$ denotes $\F_9$ and $\times$ denotes $\Z/2$) }\label{fig:eo2pic}
\end{figure}

\end{proof}

\subsection{The maximal finite subgroup}\label{sec:maxfinite}
In this short section we sketch a direct proof in the case that $G$ is a maximal finite subgroup of $\G_n$ with $p$-torsion. We include this to show that after taking Galois invariants the calculation is slightly easier; in particular, we do not need to use \Cref{topunivformula}. The calculation of $\pic(E_n^{hG})$ in this case was previously known to Hopkins. 

We start by recalling that, modulo the image of the transfer, we have
\begin{equation}\label{eq:maxfinite}
H^*(G,E_*) \simeq \F_{p}[\alpha,\beta,\Delta^{\pm 1}]/(\alpha^2)
\end{equation}
where $|\Delta| = (0,2pn^2)$, and that the homotopy groups $\pi_*(E^{hG})$ are $2p^2n^2$ periodic; this sets $\Z/(2p^2n^2)$ as a lower bound on the order of the Picard group. As usual, we will use the Picard spectral sequence to determine an upper bound.

In the range $t \ge 2$ we have classes
\[
e_c = \alpha \beta^{pc-1}\Delta^{(c(1-pn)+n-1)/n^2}, 
\]
in the 0-stem, indexed by integers $c \ge 1$, which have cohomological degree $s(e_c) = 2pc-1$. Once again, one can check that the classes $e_c$ for $c \ne 1$ do not survive the Picard spectral sequence. One key difference to the case of $C_p$ is that we see immediately that $e_1$ can contribute a group of at most order $p$ to $\pic(E^{hG})$, thus avoiding the use of~\Cref{topunivformula}. 

We are then left with calculating the contributions from $t=0$ and $t=1$. It follows
from~\cite{baker_richter_invertible} that $H^0(G,\pi_0\pics (E)) \simeq H^0(G,\Z/2) \simeq \Z/2$, and so we are left to determine $H^1(G,\pi_1\pics(E)) \simeq H^1(G,E_0^\times)$. This can be done in at least a couple of different ways; for example, as in~\Cref{prop:algebraicterm}, or using an exponential map and several Bockstein spectral sequences as in \cite{karamanov_pic}. In any case, the result is a group of order $pn^2$.

Putting all of this together, we obtain an upper bound $2 \cdot pn^2\cdot p$ for the abutment, i.e. for the Picard group of $E^{hG}$. However, by~\Cref{rem:periodicity} we know the the Picard group contains a cyclic group of order $2p^2n^2$, and so it follows that $\pic(E^{hG})$ is indeed cyclic of order $2p^2n^2$. We illustrate this for $n=2,p=3$ in \Cref{fig:eo2picmax}. 
\begin{figure}[!htbp]
\centering
\includegraphics{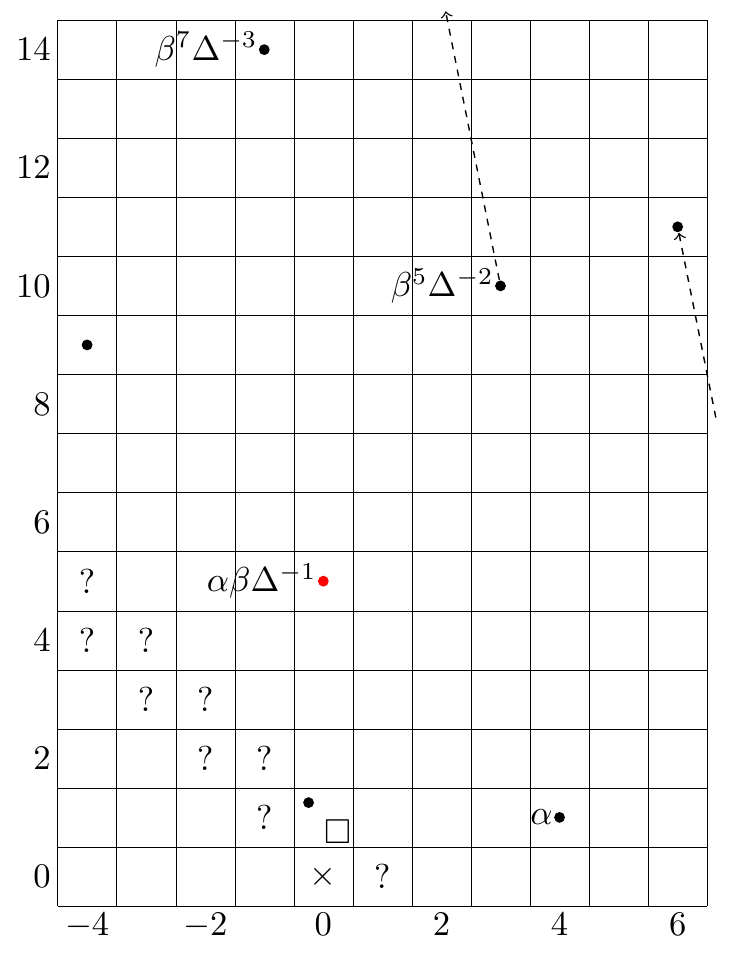}
\caption{The HFPSS for $\pi_*\pics E^{hG}$ in the case $n=2,p=3$ for $G \subset \G$ a maximal finite subgroup whose order is divisible by $p$.}\vspace{-4mm} 
\caption*{  ($\square$ denotes $\Z/4$, $\bullet$ denotes $\Z/3$, and $\times$ denotes $\Z/2$) }\label{fig:eo2picmax}
\end{figure}

\part{Computational tools}

\section{Truncated logarithms and differentials in the algebraic setting}
\label{algTruncLogsec}
\subsection{Overview}
In this section, we describe a tool for working with the following general
situation. 
Let $R$ be a complete noetherian local ring with maximal ideal $\mathfrak{m} \subset R$ such that
the residue field $R/\mathfrak{m}$ has characteristic $p$. Suppose $G$ is a finite group acting
on $R$. 
Our goal is to calculate the cohomology groups
$ H^i(G, R^{\times}).  $
We assume that the additive analog, the cohomology groups $H^i(G, R)$, are
understood by other means. 

One approach towards $H^i(G, R^{\times})$ is to consider the
$\mathfrak{m}$-adic filtration on $R$. It induces a $G$-stable filtration on $R^{\times}$,
\begin{equation} \label{multfilt} R^{\times} \supset 1 + \mathfrak{m} \supset 1 + \mathfrak{m}^2 \supset
\dots,  \end{equation}
which leads to a spectral sequence
\begin{equation} \label{multss} E_1^{i,j} =  H^i( G, (1 + \mathfrak{m}^j)/(1 + \mathfrak{m}^{j+1})) \implies H^i( G,
R^{\times}) . \end{equation}
Here the differentials run $d_k \colon E_{k}^{i,j} \to E_k^{i+1, j+k}$.
For $j > 0$, there is a $G$-equivariant isomorphism (sending $1 + x
\mapsto x$)
\[  (1 + \mathfrak{m}^j)/(1 + \mathfrak{m}^{j+1}) \simeq
\mathfrak{m}^j/\mathfrak{m}^{j+1} \]
between the associated graded
of the filtration \eqref{multfilt} and the $\mathfrak{m}$-adic filtration on
$R$ itself. 
The latter, of course, leads to a spectral sequence
(indexed similarly)
\begin{equation} 
\label{addss}
E_{1}^{i,j} = H^i(G, \mathfrak{m}^j/\mathfrak{m}^{j+1}) \implies H^i ( G, R).
\end{equation} 
In practice, one may (and is likely to) understand the additive spectral sequence \eqref{addss} concretely in specific
instances, for instance, because of the additional tools provided by
multiplicativity.
The additive and the multiplicative spectral sequences have almost the same $E_1$
pages, and in this section we will give a tool for comparing differentials in a
range of dimensions. 

\subsection{The $p$-truncated exponential}
Let $R$ be a commutative ring and let $I \subset R$
be an ideal such that $I^2 = 0$. 
Then there is an isomorphism of groups
\[ I \simeq 1 + I \subset R^{\times}, \quad x \mapsto 1 + x.  \]
We now review a crucial and elementary piece of algebra that will enable
generalizations of this fact when certain primes are inverted in $R$. 

\begin{definition} 
Let $R$ be a $\mathbb{Z}[1/(p-1)!]$-algebra. We define a function
\begin{equation} \label{def:truncexp} \exp_p \colon R \to R, \quad \exp_p(x)  =1 + x + \frac{x^2}{2!} + \dots +
\frac{x^{p-1}}{(p-1)!} ,\end{equation}
called the \emph{$p$-truncated exponential,}
and a function 
\begin{equation} 
\label{def:trunclog}
\log_p \colon R \to R, \quad \log_p(x) = \sum_{n= 1}^{p-1} (-1)^{n-1}\frac{(x-1)^n}{n},
\end{equation} 
called the \emph{$p$-truncated logarithm.}
\end{definition} 

The significance arises from the following well-known application.
\begin{lemma} 
If $I \subset R$ is an ideal with $I^p = 0$ and $R$ is a
$\mathbb{Z}[1/(p-1)!]$-algebra, then 
the map
\[ \exp_p \colon I \to 1 + I \subset R^{\times}  \]
is a group isomorphism (whose inverse is given by the truncated logarithm). 
\end{lemma} 

We note that the inverse isomorphisms $\exp_p, \log_p$ between $I$ and $1 + I$
are compatible with the $I$-adic filtration, and with the isomorphisms on associated
graded modules $ I^k/I^{k+1} \simeq (1 + I^k)/(1 + I^{k+1}) $ given by $ 1 + x \mapsto
x$.

\begin{example} 
Suppose $(A, \mathfrak{m})$ is a local ring whose residue field is of
characteristic $p$, a prime larger than $p$, or zero. Then for each $k>0$, we have a functorial isomorphism 
of groups
\[  (1  + \mathfrak{m}^k)/(1 + \mathfrak{m}^{pk}) \simeq
\mathfrak{m}^k/\mathfrak{m}^{pk},  \]
given by the truncated logarithm. 
This leads to an identification of the spectral sequences \eqref{addss} and
\eqref{multss} in a ``stable'' range of dimensions (which we will make more
precise below). 
\end{example}

We will actually need a slight generalization of this fact (which measures
the first-order failure of $\exp_p$ to be a homomorphism in general).

\begin{definition} 
Given $x_1, \dots, x_k \in R$, we define
\begin{equation} \label{def:sigma} \sigma_p(x_1, \dots, x_k)
\stackrel{\mathrm{def}}{=} \frac{( \sum_i
x_i )^p - \sum_i x_i^p}{p!} =  
\sum_{\substack{j_1, \dots j_k \in [0, p-1] \\
j_1 + \dots + j_k = p
}}^{p-1} 
\frac{x_1^{j_1} x_2^{j_2} \dots x_k^{j_k}}{j_1!\dots j_k!}
 , \end{equation}
 so that we obtain a function $\sigma_p( \cdot, \dots, \cdot)\colon R^k \to R$. 
 If $M$ is a $\mathbb{Z}[1/(p-1)!]$-module, we will also use $\sigma_p$ to
 denote the induced function
 \[ \sigma_p \colon M^k \to \sym^p M,  \]
 defined by the same formula. 
\end{definition}

\begin{prop} 
Let $R$ be a $\mathbb{Z}[1/(p-1)!]$-algebra. Suppose $x_1, \dots, x_k \in R$ and let $I = (x_1, \dots, x_k) \subset R$. Then 
\begin{equation} \label{expsumformula} \prod \exp_p(x_i) - \exp_p(\sum x_i) \equiv \sigma_p(x_1,
\dots, x_k) \mod I^{p+1}. \end{equation} 
\end{prop} 
\begin{proof} 
This follows by a straightforward computation. 
\end{proof} 
With notation as above, observe that $\sigma_p(x_1,\dots, x_k) \in I^p$ and
that $\exp_p(x) \in 1 + I$ for $x \in I$. As a result, we also find that, if $I$ is
contained in the Jacobson radical of $R$, then we have the formula
\begin{equation} \label{quotexpformula1}\frac{\exp_p(\sum x_i)}{\prod \exp_p(x_i)} \equiv 1 - \sigma_p(x_1, \dots,
x_k) \mod I^{p+1}.  \end{equation}
For example, taking $k = 2$ and taking $x_1 = y, x_2 = -y$, we find
\begin{equation} 
\label{invequation}
\exp_p( y) \exp_p(-y) \equiv 1 + \sigma_p( y, -y) \mod I^{p+1}.
\end{equation} 
Note in addition that $\sigma_p(y, -y) = y^p \sigma_p( 1, -1)$. 

\begin{cor}
If $y_0, \dots, y_n \in I$ and $I$ is contained in the radical of $R$,
then
\begin{equation} 
\label{alternatingsumthm}
\frac{\exp_p( \sum_{i = 0}^n (-1)^i y_i)}{\prod_{i=0}^n \exp_p(y_i)^{(-1)^i}} \equiv 
1 - \sigma_p( y_0, -y_1, \dots , (-1)^n y_n) + \sum_{i \ \mathrm{odd}}
\sigma_p(y_i, -y_i)  \quad \mod I^{p+1}.
\end{equation} 
\end{cor}
\begin{proof} 
This follows from \eqref{quotexpformula1} taking $x_i = (-1)^i y_i$. 
We multiply it by 
$\prod_{i \ \mathrm{odd}}\left( \exp_p( -y_i) / \exp_p(y_i)^{-1}\right)$ 
and invoke \eqref{invequation}. 
\end{proof} 

We give a name to the discrepancy term; given $y_0, \dots, y_n$, we set
\begin{equation} \mu_p( y_0, \dots, y_n) \stackrel{\mathrm{def}}{=}  
\sigma_p( y_0, -y_1, \dots , (-1)^n y_n) - \sum_{i \ \mathrm{odd}}
\sigma_p(y_i, -y_i).
\end{equation}

\subsection{The stable range and first differential}

We begin by describing a version of the general setup of comparison of spectral
sequences from the introduction to this section.

\begin{prop} 
\label{comparediffs}
Let $G$ be a finite group acting on a local noetherian ring $(R,
\mathfrak{n})$ with $(p-1)!$ invertible. 
Consider the two \emph{truncated} spectral sequences
\begin{align*}
 E_1^{s,t} &= H^s(G, \mathfrak{n}^t/\mathfrak{n}^{t+1}) \implies H^i(G,
\mathfrak{n}^k/\mathfrak{n}^{pk}), \\
E_{1 }^{s,t} &= H^i(G, \mathfrak{n}^t/\mathfrak{n}^{t+1}) \implies H^i(G,
(1 + \mathfrak{n}^k)/(1  +\mathfrak{n}^{pk}))
,
\end{align*}
where $t \in [k, pk-1]$. Then the natural identification of $E_1$-terms extends
to an isomorphism of spectral sequences given by the $p$-truncated exponential. 
\end{prop}

We now describe the setup that will enable us to ultimately give a formula for
the first unequal differential. 
For convenience, we will work with \emph{nonunital} commutative rings here. 
If $A$ is a nonunital commutative ring, we define an abelian monoid $1 + A$
consisting of formal elements $1 + x, x \in A$ with the obvious
multiplication rule. 
\begin{cons}
Suppose $\mathfrak{m}^\bullet$ is a \emph{nonunital} cosimplicial
commutative $\mathbb{Z}[1/(p-1)!]$-algebra such that $\mathfrak{m}^{\bullet,
p+1} = 0$.\footnote{The notation $\mathfrak{m}^{\bullet,k}$ means the $k$-th power of the cosimplicial ring $\mathfrak{m}^\bullet$.}
We have a descending multiplicative filtration by powers
$$\mathfrak{m}^{\bullet} \supset \mathfrak{m}^{\bullet, 2} \supset \dots .$$
Consider also the cosimplicial abelian group $1 +
\mathfrak{m}^\bullet$.\footnote{Since $\mathfrak{m}^{\bullet}$ is nilpotent,
this is a cosimplicial abelian group  and not only a cosimplicial abelian monoid.} 
We also have a filtration of cosimplicial abelian groups
\[ 1 + \mathfrak{m}^\bullet  \supset 1 + \mathfrak{m}^{\bullet,2} \supset \dots,  \]
and the successive quotients are isomorphic to those of the additive one. 

Both lead to spectral sequences starting from the associated graded
$\mathfrak{m}^{\bullet, j}/ \mathfrak{m}^{\bullet , j+1}$ (for $1 \leq j \leq
p$) 
converging to the cohomology of $\mathfrak{m}^\bullet$ (resp. $1 +
\mathfrak{m}^\bullet$), i.e., the $E_1$-pages are identified.  
We have spectral sequences
\begin{gather} E_1^{i,j} = H^i( \mathfrak{m}^{\bullet, j}/\mathfrak{m}^{\bullet , j+1})
\implies    H^i( \mathfrak{m}^{\bullet}) \\
E_1^{i,j ,\times} = H^i( \mathfrak{m}^{\bullet, j}/\mathfrak{m}^{\bullet , j+1})
\implies    H^i( 1 + \mathfrak{m}^{\bullet}).
\end{gather}
The $E_1$-terms are identified. We will denote the differentials at the $k$th page
by $d_k$ and $d_{k}^{\times}$, respectively. 

Using the truncated exponential, we obtain isomorphisms
of filtered cosimplicial abelian groups
\[ \mathfrak{m}^{\bullet}/\mathfrak{m}^{\bullet, p} \simeq (1 +
\mathfrak{m}^\bullet)/(1 + \mathfrak{m}^{\bullet, p}), \quad
\mathfrak{m}^{\bullet, 2} \simeq 1 + \mathfrak{m}^{\bullet, 2}.
\]
inducing the above isomorphism on the associated graded. 
\end{cons} 
The above construction yields the following result.
\begin{prop}[Spectral sequence comparison result, stable range] 
In the spectral sequences $E_k^{\ast, \ast}$ and $E_k^{\ast, \ast, \times}$
above, we have isomorphisms $E_k^{\ast, \ast} \simeq E_k^{\ast, \ast,
\times}$  for $ k \leq
p-1$.
This identification is
compatible with the above identification at $k = 1$
 and is compatible with the differentials $d_k$ for $k < p-1$. 
\end{prop}

The most important goal is to obtain a description of the first
``unstable''
differential $d_{p-1}$ (in the present setup of nonunital algebras with
$(p+1)$-fold products zero, it is also the last differential
in the spectral sequence). 

Let $P^\bullet$
be a cosimplicial $\mathbb{Z}[1/(p-1)!]$-module. Then we can consider 
the cosimplicial $\mathbb{Z}[1/(p-1)!]$-module
$\sym^p P^\bullet$ obtained by applying the $p$th symmetric power levelwise. 
We construct an operation
\[ H^i(P^\bullet) \to H^{i+1}( \sym^p P^\bullet).  \]

\begin{cons} 
\label{constructionofphi}
Let $x \in P^i$. 
We define
\[ \phi(x) \stackrel{\mathrm{def}}{=}\mu_p( \partial^0 x , \partial^1 x, \partial^2 x, \dots,  \partial^i x) \in \sym^p
P^{i+1}.   \]
Here the $\partial^k$ are the cosimplicial operators. 
We claim that: 
\begin{enumerate}
\item If $x$ is a cocycle in $P^i$, then $ \phi(x)$ is a cocycle in $\sym^p P^{i+1}$. 
\item If $x,y \in P^{i}$ are cohomologous cocycles, then $\phi(x)$
and $\phi(y)$ are cohomologous in $\sym^p P^{i+1}$. 
\end{enumerate}
The proofs of both these claims 
will be contained in \Cref{firstalgdiff} below. 
As a result, we will also abuse notation and use $\phi$ to denote the induced
map in cohomology. 
\end{cons}

Return to the previous spectral sequences for computing the cohomology of the
cosimplicial abelian groups $\mathfrak{m}^\bullet$ and $1 +
\mathfrak{m}^\bullet$. 
The $(p-1)$th differential in each spectral sequence, namely $d_{p-1}$ or $ d_{p-1, \times}$, runs from a sub-object of $H^i(
\mathfrak{m}^\bullet/\mathfrak{m}^{ \bullet, 2})$ to a quotient of $H^{i+1}(
\mathfrak{m}^{\bullet, p})$. 
\begin{prop}[First unstable differential]
\label{firstalgdiff}
The operator $\phi$ satisfies the claimed properties (1) and (2) above, and 
we have $d_{p-1, \times}= d_{p-1} + \phi$. 
\end{prop} 

Let us explain this formula. We have a natural (multiplication) morphism of cosimplicial
abelian groups $\sym^p (
\mathfrak{m}^\bullet/\mathfrak{m}^{\bullet, 2}) \to \mathfrak{m}^{\bullet , p}$. 
The operator $\phi$ gives us  a natural map from $H^i(
\mathfrak{m}^\bullet/\mathfrak{m}^{ \bullet, 2})$ to 
$H^{i+1}( \sym^p (\mathfrak{m}^\bullet/\mathfrak{m}^{ \bullet, 2}))$ and we
compose that with the multiplication map.  

\begin{proof} 
We unwind the definitions. A class $x \in H^i( \mathfrak{m}^\bullet/\mathfrak{m}^{
\bullet, 2})$ that has survived to the $E_{p-1}$ stage of the spectral sequence, by
definition, lifts to a class  in $H^i(
\mathfrak{m}^\bullet/\mathfrak{m}^{\bullet, p})$, which we represent by a
degree $i$ cycle $\widetilde{x}$. 
The short exact sequence of cosimplicial abelian groups
\[ 0 \to  \mathfrak{m}^{ \bullet, p} \to \mathfrak{m}^\bullet \to
\mathfrak{m}^\bullet/\mathfrak{m}^{ \bullet, p} \to 0,   \]
leads to a connecting homomorphism 
\[ \delta \colon  H^i(\mathfrak{m}^\bullet/\mathfrak{m}^{ \bullet, p}) \to
H^{i+1}(\mathfrak{m}^{ \bullet, p}).  \]
The differential $d_{p-1}(x)$ is given by the image of $\delta( \widetilde{x})$. 

In the multiplicative setting, we know that the cycle $\exp_p( \widetilde{x}) \in (1 +
\mathfrak{m}^\bullet)/(1 + \mathfrak{m}^{\bullet, p})$ lifts the image of $x$
(or rather, $1  + x$)
in the $E_1$-page of the multiplicative spectral sequence. 
Now $d_{p-1, \times}(x)$ is given as follows: one considers the short exact
sequence of cosimplicial abelian groups
\[ 0 \to  (1 + \mathfrak{m}^{ \bullet, p}) \to (1 + \mathfrak{m}^\bullet )\to
(1 + \mathfrak{m}^\bullet)/(1 + \mathfrak{m}^{\bullet, p}) \to 0 ,   \]
and the induced connecting homomorphism 
\[ \delta_{\times} \colon H^i (
(1 + \mathfrak{m}^\bullet)/(1 + \mathfrak{m}^{\bullet, p})) 
\to H^{i+1}(1 + \mathfrak{m}^{\bullet, p}  ) .
\]
Unwinding the definitions, the multiplicative differential $d_{p-1, \times}(x)$
is given by 
$$d_{p-1, \times}(x) = \delta_{\times}( \exp_p( \widetilde{x})) \in H^{i+1}(1 + \mathfrak{m}^{
\bullet, p}  ) \simeq H^{i+1}( \mathfrak{m}^{ \bullet, p}),$$
or rather the image in the quotient of this at the $E_{p-1}$-page.

It follows now that we need to compare the additive and the multiplicative
boundary maps $\delta, \delta^{\times}$. 
Given a cycle 
$\widetilde{x} \in \mathfrak{m}^\bullet/\mathfrak{m}^{\bullet,
p}$
in degree $i$,  projecting to a cycle and induced cohomology class $x \in \mathfrak{m}^{\bullet}/\mathfrak{m}^{\bullet,
2}$, we claim that
\begin{equation} \delta_{\times}( \exp_p(\widetilde{x})) = 
1 + \delta( \widetilde{x}) + \phi( \widetilde{x}) \in H^{i+1}( 1 + \mathfrak{m}^{\bullet,
p}) \simeq H^{i+1}( \mathfrak{m}^{\bullet, p}).
\end{equation}

To prove this, we unwind the definitions again. Lift $\widetilde{x}$ to a
\emph{chain} $x' \in \mathfrak{m}^\bullet$ (in degree $i$). 
The additive connecting homomorphism $\delta( \widetilde{x})$ is obtained by
considering the image in cohomology of the 
alternating sum $y' = \sum_k ( -1)^k
\partial^k x'$, which is an $(i+1)$-cocycle in $\mathfrak{m}^{\bullet, p}$. 
For
the multiplicative version, we have to consider
the cocycle
\[ \delta_{\times}( \exp_p( \widetilde{x})) = \prod_k (\partial^k \exp_p( x'))^{(-1)^k} = \prod_k  \exp_p( \partial^k
x')^{(-1)^k} \in 1 + \mathfrak{m}^{\bullet, p+1}.   \]
Now we use \eqref{alternatingsumthm}, which gives us 
\begin{align*} \delta_{\times}( \exp_p( \widetilde{x})) & = 
\exp_p\left( \sum (-1)^k \partial^k x' \right) + \mu_p( \partial^0 x',  \partial^1
x', \dots,  \partial^i x') \\
& = \exp_p( y') + \mu_p( \partial^0 x',  \partial^1
x', \dots,  \partial^i x') \\
& = 1 + y' + \mu_p( \partial^0 x',  \partial^1
x', \dots,  \partial^i x') .
\end{align*}
In cohomology, this is the desired claim, provided we can justify the claims
about $\phi$ made in \Cref{constructionofphi}. 

The justification of \Cref{constructionofphi} now follows from reversing the
above argument. Keep the notation from there. 
We consider the commutative nonunital
cosimplicial algebra
\[ A^\bullet = \bigoplus_{0 < k \leq p} \sym^k P^\bullet 
\]
and its quotient $\overline{A}^\bullet = \bigoplus_{0 < k < p} \sym^k
P^\bullet$. Choose a cycle $x \in P^\bullet \subset \overline{A}^\bullet$,
which clearly extends to a cycle in $A^\bullet$. 
Then $\exp_p( x)$ is a cycle in $1 + \overline{A}^\bullet$. 
The coboundary of this class in $\sym^p P^\bullet \simeq 1 + \sym^p P^\bullet
\subset 1 + A^\bullet$
is given by $\phi(x)$ by the above analysis. This shows that $\phi(x)$ is a
cycle and that its cohomology class depends only on that of $x$. 
\end{proof}

\subsection{Identification with $\beta \mathcal{P}^0$}
	
Let $P^\bullet$ be a cosimplicial $\mathbb{Z}[1/(p-1)!]$-module. 
In the previous subsection, we 
described a natural operation (based on a formula)
\begin{equation}  \phi \colon H^i( P^\bullet) \to H^{i+1}( \sym^p P^\bullet).
\end{equation} 

Suppose $i > 0$.\footnote{When $i = 0$, the power operation $\beta \mathcal{P}^0 = 0$
and the operation $\phi$ is identically zero, so we shall not discuss this case.} 
In the setting of cosimplicial $\mathbb{F}_p$-vector spaces, 
the work of Priddy \cite{priddy} has classified all such operations, which
amount to determining the cohomology of the symmetric powers of a
cosimplicial vector space (one reduces to an Eilenberg--MacLane complex).  
In particular, for cosimplicial $\mathbb{Z}[1/(p-1)!]$-modules, 
the only such natural operations are the scalar multiples of a single operation
called $\beta \mathcal{P}^0$, as one sees from Priddy's work (and the
Bockstein spectral sequence).
In other words, in the ``universal'' case when $H^*(P^\bullet)$ is a
$\mathbb{Z}[1/(p-1)!]$ (generated by $\iota$) concentrated in degree $i$, we have $H^{i+1} ( \sym^p
P^\bullet) \simeq \mathbb{Z}/p$, generated by the class $\beta \mathcal{P}^0
\iota$.

Given a cosimplicial $\mathbb{F}_p$-vector space $P^\bullet$, the cosimplicial
commutative ring $\sym^* P^\bullet$ yields (upon taking the totalization) an
$\einfty$-algebra over $\mathbb{F}_p$. 
Then $\beta \mathcal{P}^0$ is identified with the power operation on
$\einfty$-algebras over $\mathbb{F}_p$ with the same name. 

In this section, we will prove: 
\begin{prop} 
\label{betaP0phi}
For cosimplicial $\mathbb{Z}[1/(p-1)!]$-modules, we have
$\phi = -\beta \mathcal{P}^0$.
As a result, in the situation of \Cref{firstalgdiff}, we have
\begin{equation} d_{p-1, \times} = d_{p-1} - \beta \mathcal{P}^0.
\end{equation}
\end{prop} 

While one can compare explicit formulas, we will approach the determination of
the operator $\phi$  by considering examples. 
We note that this formula will not be used in the sequel. Rather, it will
suffice to know that $\phi$ is Frobenius-semilinear.

First of all, we give an example to show that $\phi$ is not identically zero,
i.e., that the additive and multiplicative versions of $d_{p-1}$
can in fact differ. 
\begin{example} 
Consider the commutative algebra 
$\mathbb{F}_p[x]/x^{p+1}$ and the ideal $\mathfrak{m} = (x)$, so that
$\mathfrak{m}^{p+1} = 0$. 
We then have an isomorphism 
\[ (1 + \mathfrak{m})/(1 + \mathfrak{m}^{p}) \simeq (\mathbb{Z}/p)^{p-1},  \]
as one sees via the truncated exponential and logarithm. However, we also have 
\[ 1 + \mathfrak{m} \simeq \mathbb{Z}/p^2 \oplus (\mathbb{Z}/p)^{p-2}. \]
In fact, the unit $1+x \in  1 + \mathfrak{m}$ has order exactly $p^2$.
Meanwhile, $\exp_p(x^2), \dots, \exp_p(x^{p-1})$ each have order $p$ (one
could also take $1 + x^2, \dots, 1 + x^{p-1}$).  
One sees that the map 
\[ 1 + \mathfrak{m}  \to (1 + \mathfrak{m})/(1 + \mathfrak{m}^p) \]
can be identified with the direct sum of the reduction map $\mathbb{Z}/p^2 \to
\mathbb{Z}/p$ and the identity $(\mathbb{Z}/p)^{p-2} \to (\mathbb{Z}/p)^{p-2}$.

Now let $X_\bullet $ be a simplicial set and suppose that there exists a class
in $H^n(X_\bullet, \mathbb{Z}/p)$ which does not lift to $H^n(X_\bullet,
\mathbb{Z}/p^2)$. 
Form the nonunital cosimplicial commutative algebra $F(X_\bullet,
\mathfrak{m}) = \mathfrak{m}^{X_\bullet}$, whose group of units is $F(X_\bullet, 1 + \mathfrak{m})$. We can apply the machinery of the previous subsection to this example,
where the filtration in question is the $\mathfrak{m}$-adic filtration. 

The additive spectral sequence converges to the cohomology of $X_\bullet$ with
coefficients in $(\mathbb{Z}/p)^p$ and the multiplicative spectral sequence
converges to the cohomology with coefficients in $\mathbb{Z}/p^2 \oplus
(\mathbb{Z}/p)^{p-2}$. 
We see in
particular that, while there are no differentials additively (the filtration
splits), there is necessarily a $d_{p-1, \times}$ since there are cohomology
classes in $H^n(X_\bullet, \mathbb{Z}/p)$ which do not lift to $H^n(X_\bullet,
\mathbb{Z}/p^2)$.
\end{example}

We now describe an example where the additive spectral sequence supports a
differential but the multiplicative spectral sequence does not. 
\begin{example} 
Consider the complete local ring $R = \mathbb{Z}_p[\zeta_p]$ and the maximal
ideal $\mathfrak{m} = (1  - \zeta_p)$. 
We consider the quotient $\overline{R} = R/\mathfrak{m}^{p+1}$ and the image
$\overline{\mathfrak{m}}$ of $\mathfrak{m}$. 
We have: 
\begin{enumerate}
\item $(1 + \overline{\mathfrak{m}})/(1 + \overline{\mathfrak{m}}^p) \simeq
\overline{\mathfrak{m}}/\overline{\mathfrak{m}^p}  \simeq
(\mathbb{Z}/p)^{p-1}$. (Note here that $p \in\mathfrak{m}^{p-1}\setminus
\mathfrak{m}^p$.)
\item $\overline{\mathfrak{m}} \simeq \mathbb{Z}/p^2 \oplus
(\mathbb{Z}/p)^{p-2}$. The generator of the $\mathbb{Z}/p^2$ is $(1 - \zeta_p)$
and the generators of the remaining terms are $(1 - \zeta_p)^i, i = 2, \dots,
p-1$.
The map $\overline{\mathfrak{m}} \to
\overline{\mathfrak{m}}/\overline{\mathfrak{m}^p}$ is a direct sum of
the reduction $\mathbb{Z}/p^2 \twoheadrightarrow \mathbb{Z}/p$ and the identity
$(\mathbb{Z}/p)^{p-2} \to (\mathbb{Z}/p)^{p-2}$.
\item $1 + \overline{\mathfrak{m}} \simeq (\mathbb{Z}/p)^{p}$. 
In fact, one sees that every element has order $p$ here (since $\zeta_p$ does). 
The map $1 + \overline{\mathfrak{m}} \to (1 + \overline{\mathfrak{m}})/(1 +
\overline{\mathfrak{m}}^p)$ is a surjection $(\mathbb{Z}/p)^p \to
(\mathbb{Z}/p)^{p-1}$.
\end{enumerate}
Let $X_\bullet$ be a simplicial set as in the previous example. 
In this case, reasoning analogously, we find that the \emph{additive} spectral sequence for 
$F(X_\bullet, \overline{\mathfrak{m}})$
must support a
$d_{p-1}$ but the multiplicative spectral sequence does not. 
\end{example}

\begin{proof}[Proof of \Cref{betaP0phi}] 
We will analyze the last example more carefully to determine the precise
multiple.\footnote{We note that the precise multiple will not be used in the
sequel, only that it is not zero.} 
Note that $\overline{\mathfrak{m}}/\overline{\mathfrak{m}}^2 \simeq
\mathbb{F}_p$ (generated by the image of $1 - \zeta_p$) and 
$\overline{\mathfrak{m}}^p\simeq
\mathbb{F}_p$ (generated by the image of $(1 - \zeta_p)^p$). 
The (additive) differential $d_{p-1}$ comes from the short exact sequence
\begin{equation} \label{thisses} 0 \to \overline{\mathfrak{m}}^p \to \overline{\mathfrak{m}} \to
\overline{\mathfrak{m}}/\overline{\mathfrak{m}}^p \to 0,  \end{equation}
and we have a commutative diagram
of short exact sequences
\begin{equation}\label{natural:ses} \xymatrix{
0 \ar[r] &  \mathbb{Z}/p \ar[r]^{1 \mapsto p} \ar[d]  &  \mathbb{Z}/p^2 \ar[r] \ar[d]  &
\mathbb{Z}/p \ar[d] \ar[r]
&  0 \\
0 \ar[r] & \overline{\mathfrak{m}}^p \ar[r] &  \overline{\mathfrak{m}} \ar[r]
&  \overline{\mathfrak{m}}/\overline{\mathfrak{m}}^p \ar[r] &  0
}.\end{equation}
Here the map $\mathbb{Z}/p^2 \to \overline{\mathfrak{m}}$ carries $1 \mapsto
1 - \zeta_p$. 
Note that $p(1 - \zeta_p) \equiv (1 - \zeta_p)^p $ in $\overline{\mathfrak{m}}$
as one sees from the expression $\prod_{i=1}^{p-1} (T - \zeta_p^i) = T^{p-1} +
T^{p-2} + \dots + 1$ with $T =1$. 

Take $X_\bullet = K( \mathbb{Z}/p, n)$, so that $H^n( X_\bullet, 
\mathbb{F}_p) = H^{n+1}(X_\bullet, \mathbb{F}_p) = \mathbb{F}_p$.
We let $\iota$ be a generator of the former, so that $\beta \iota$ is a
generator of the latter for $\beta$ the Bockstein.
Again, this analysis is based on the cosimplicial commutative ring
$\overline{\mathfrak{m}}^{X_\bullet}$. 
\begin{enumerate}
\item Note that  $H^n(X_\bullet, \overline{\mathfrak{m}}/\overline{\mathfrak{m}}^p)
\simeq H^n(X_\bullet, \mathbb{F}_p)
\otimes_{\mathbb{F}_p}\overline{\mathfrak{m}}/\overline{\mathfrak{m}}^p \simeq
\overline{\mathfrak{m}}/\overline{\mathfrak{m}}^p \iota$ and similarly for
$H^{n+1}(X_\bullet, \overline{\mathfrak{m}}^p)$. The differential in question 
comes from the connecting homomorphism in the short exact sequence
\eqref{thisses}, so that by naturality of \eqref{natural:ses}, it carries $(1 - \zeta_p) \iota$ to $(1 -
\zeta_p)^p \beta\iota \in H^{n+1}(X_\bullet, \overline{\mathfrak{m}^p})$.
\item The associated graded of $\overline{\mathfrak{m}}$ comes from the
reduction of a $\mathbb{Z}$-algebra, so $\beta \mathcal{P}^0$ is the composite
of the Bockstein and $\mathcal{P}^0$. Now $\mathcal{P}^0( (1 - \zeta_p)
\iota) = (1 - \zeta_p)^p \iota \in H^n( X, \overline{\mathfrak{m}}^p)$ since
$\mathcal{P}^0$ is the Frobenius on underlying rings. Applying $\beta$ to that,
we get $(1 - \zeta_p)^p \beta \iota$. 
\end{enumerate}
Putting these together, it follows that if the class corresponding to $(1 -
\zeta_p) \iota$ is to survive in the multiplicative spectral sequence (as it
must), the formula must be as desired. 
\end{proof} 

We refer to \cite[Lecture 27]{Mumford} for an instance of the appearance of
these Bockstein operators in the deformation theory of the Picard scheme in
characteristic $p$.
In that case, there are no algebraic obstructions, but the multiplicative
operations are controlled by the Bocksteins.

\section{The algebraic Picard group for the $C_p$-action on $E_n$}\label{sec:algPic}

We keep the previous notation and fix a prime number $p$, setting $n = p-1$.
Let $C_p \subset \mathbb{G}$ be a maximal finite $p$-subgroup. 
Recall that $(E_n)_0 \simeq \W( \mathbb{F}_{p^n})[\![u_1, \dots, u_{n-1}]\!]$
is a complete local ring; we denote the maximal ideal by $\mathfrak{m}$. It
acquires a $C_p$-action by restriction.

In this section, we apply the tools of
the preceding section and calculate the algebraic
Picard group $H^1(C_p, (E_n)_0^{\times})$, showing that it is cyclic of
order $p$. This will be a key input into the topological descent spectral
sequence for Picard groups. 
The approach will be to use the $\mathfrak{m}$-adic filtration on
$(E_n)_0^{\times}$. 

\subsection{The additive spectral sequence}

The associated graded of the $\mathfrak{m}$-adic filtration 
of $(E_n)_0$ has been determined in unpublished work of Hopkins--Miller
\cite{hopkins_miller} and
studied additionally in work of Hill \cite{Hillthesis}. 

Since $(E_n)_0$ is a regular local ring, one has canonical and
$C_p$-equivariant isomorphisms
\[ \mathfrak{m}^i/\mathfrak{m}^{i+1} \simeq \sym^i (\mathfrak{m}/\mathfrak{m}^2)
, \quad i \geq 0,\]
where the symmetric powers are taken over the residue field
$\mathbb{F}_{p^n}$ (upon which the group $C_p$ acts trivially). 
The first main result is the following, which is a consequence of the
coordinates given on \cite[p. 497]{nave_smith_toda}. 
\begin{prop}[Hopkins--Miller] 
As a $C_p$-representation, $\mathfrak{m}/\mathfrak{m}^2$ is the reduced regular
representation $\overline{\rho}$ over $\mathbb{F}_{p^n}$. 
\end{prop}

Recall that the \emph{reduced regular representation} $\overline{\rho}$ of $C_p$ over a field $k$
 is the quotient of the regular representation $k[C_p]$ by
the trivial subrepresentation. 
As a corollary, we find that the associated graded of $(E_n)_0$, as a
$C_p$-representation and algebra, is the symmetric algebra
\[ \mathrm{gr}  (E_n)_0 \simeq \sym^* \overline{\rho}  \]
of the reduced regular representation $\overline{\rho}$. 

Over a field $k$ of characteristic $p$, 
the decomposition of the symmetric powers
of $\overline{\rho}$ as $C_p$-representations up to free summands is
calculated in work of Almkvist and Fossum (and was done by
Hopkins--Miller in \cite[Lem.
7.2.3]{hopkins_miller}).
We will let $\mathbf{1}$ denote the trivial representation.

\begin{prop}[{Almkvist--Fossum \cite[Ch.~3, Props.~3.4--3.6]{AF78}}] 
The $C_p$-representation
$\sym^i \overline{\rho}$ decomposes as follows: 
\begin{enumerate}
\item  If $i \equiv 0 \mod p$,
then $\sym^i \overline{\rho} \simeq \mathbf{1}  \oplus \mathrm{free}$.
\item If 
$i \equiv 1 \mod p$, then $\sym^{i} \overline{\rho} \simeq \overline{\rho}
\oplus \mathrm{free}$.
\item 
If $i \not \equiv 0, 1 \mod p$, then $\sym^i \overline{\rho}$ is a free
$k[C_p]$-module.
\end{enumerate}
\end{prop}
The invariant element of $\sym^p \overline{\rho}$ that generates the
$\mathbf{1}$ summand is given as follows. The representation $\overline{\rho}$
is generated by elements $\epsilon_1, \dots, \epsilon_p$ with $\epsilon_1 + \dots + \epsilon_p =
0$ and with $C_p$ acting on the $\epsilon_i$ via the natural permutation action. 
Then the (norm) product $\epsilon_1 \dots \epsilon_p \in \sym^p \overline{\rho}$ is $C_p$-invariant and
generates the desired summand. 
The analogous invariant elements  in $\sym^{pk} \overline{\rho}$ are the powers
of this norm product.

We now describe the cohomology of the associated graded of $(E_n)_0$. 
Recall that $$H^*(C_p, \mathbb{F}_{p^n} ) \simeq \mathbb{F}_{p^n}[a, b]/a^2,
\quad |a| = 1, |b|  = 2,$$
is the tensor product on an exterior algebra on a class in degree one and a
polynomial algebra on a class in degree two, both over $\F_{p^n}$. 
Recall also that the cohomology of $\overline{\rho}$ as a graded module is given by 
\[ H^*( C_p, \overline{\rho}) \simeq H^*(C_p,
\mathbb{F}_{p^n})\left\{u_0, u_1\right\}/ ( a u_0 = 0, b u_0 =
a u_1).   \]

As a result, 
we can completely determine the cohomology modulo transfers (which only appear
in $H^0$).

\begin{cor} 
The cohomology $H^*(C_p, \mathrm{gr}(E_n)_0)$ modulo transfers is given
by 
\[ H^*(C_p, \mathbb{F}_{p^n}) [u, v, v']/( v^2 = v'^2 = 0,  a v =
0, b v = a v').  \]
The internal grading of $u$ is $p$ and that of $v, v'$ is $1$. Moreover, $u, v$
are in cohomological degree zero and $v'$ is in $H^1$.
\end{cor} 

We display in \Cref{addspectralsequenceEN} the $E_1$-page of the spectral 
sequence
\begin{equation} \label{algss} H^i( C_p, \mathrm{gr}^j((E_n)_0)) \implies
H^{i}( C_p , (E_n)_0)  \end{equation}
when $p=3$.

We note that none of the transfers support 
differentials in the spectral sequence thanks to the norm map. Moreover, this
holds in the topological spectral sequence for the same reason. 

It remains to determine the differentials. 
We do this by reverse-engineering the output of the Hopkins--Miller calculation. 
\begin{prop} \label{calcofdiffs}
The spectral sequence \eqref{algss} collapses at the $p$-th page. The differentials are
(up to nonzero scalars) determined by $d_1(a) \dequal b v$ and $d_{p-1}( v')
\dequal b u$. The classes $u, b, v$ are permanent cycles.
\end{prop}

\begin{proof} 
We start by describing the permanent cycles.
First, the class $u$ in degree $p$ is constructed as a multiplicative norm. This
construction can be carried out in $(E_n)_0$, so $u$ is necessarily a
permanent cycle. The class $b$ comes from $H^*(C_p, \mathbb{Z})$ and is
also therefore a permanent cycle. The invariant class $v$ in $H^0(C_p,
\overline{\rho})$ is represented by the image of $p$ in
$\mathfrak{m}/\mathfrak{m}^2$. Since this is invariant in $(E_n)_0$, it is
also a permanent cycle.   

Next, we describe the first differential $d_1$. We claim that, up to unit scalars, $d_1( a) = b v = a v'$. 
One sees this by naturality with the filtration (by powers of the maximal
ideal) on $\W( \mathbb{F}_{p^n})$. We have a $C_p$-equivariant, filtration-preserving map
\( \W(\mathbb{F}_{p^n}) \to (E_n)_0,  \)
where the domain has trivial action. 
One easily computes the $d_1$ in the spectral sequence for $\W(\F_{p^n})$ (as a
Bockstein), and uses the fact that the unique nontrivial map of
$C_p$-representations $\mathbf{1} \to \overline{\rho}$ induces an
isomorphism on $H^2$. By naturality, this determines the $d_1$ in our spectral
sequence. 

Finally, we claim that $d_{p-1}(v') \dequal b u$ as desired. We
know that this differential must happen, because otherwise $H^2(  C_p,
(E_n)_0)$ would  have cardinality at least $p^{2n}$, whereas it has
cardinality $p^n$ by the Hopkins--Miller calculations 
(\Cref{prop:e2cals}). 
\end{proof} 
\begin{figure}[!hbt]

\begin{centering}
\includegraphics{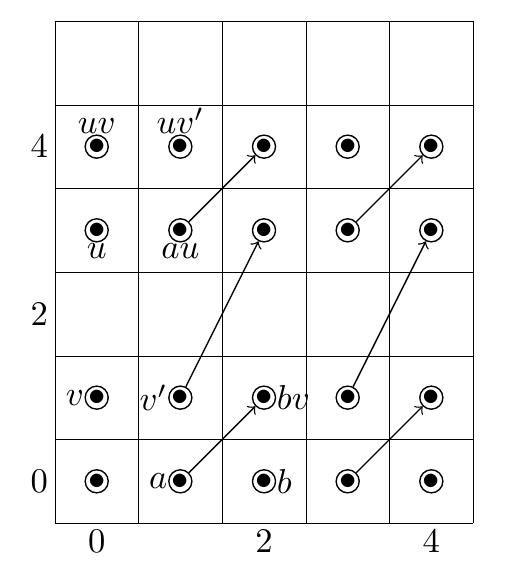}
\end{centering}
\caption{The spectral sequence for $H^*(C_p, (E_n)_0)$ for $p = 3$}
\label{addspectralsequenceEN}
\end{figure}

\subsection{The multiplicative spectral sequence}

We now describe the modifications to the above spectral sequence that occur for
the multiplicative case, and determine the algebraic Picard group for the group
$C_p$. 

The following is our main result. As before, we let $\mathfrak{m} \subset
(E_n)_0$ be the maximal ideal, with the induced $C_p$-action, and let $\omega$
denote $\pi_2E_n$ as an (invertible) equivariant $(E_n)_0$-module.

\begin{prop} \label{prop:algebraicterm}
$H^1(C_p, 1 + \mathfrak{m}) \simeq H^1(C_p, (E_n)_0^{\times}) \simeq
\mathbb{Z}/p$, generated by the class of $\omega$. 
\end{prop} 

\begin{proof} 
We use the $\mathfrak{m}$-adic filtration on $1 + \mathfrak{m}$, which leads to
a spectral sequence whose $E_1$-page is identified with \eqref{algss} with the
bottom row removed, converging to $H^*(C_p, 1 + \mathfrak{m})$. We will denote
the differentials by $d_{k, \times}$.

In the \emph{algebraic} spectral sequence, we saw in 
\Cref{calcofdiffs} that 
everything  that can contribute to the
cohomological degree $1$ abutment supports either a
$d_1$ or a $d_{p-1}$. In filtration $\geq 2$, we can import these differentials
thanks to the truncated exponential and logarithm (\Cref{comparediffs}). 
However, we \emph{cannot} import the formula
\[ d_{p-1}(v') \dequal b u,  \]
into the multiplicative spectral sequence. Instead, we use the formula of
\Cref{betaP0phi},  
which gives the differential in the multiplicative spectral sequence. 
Suppose $d_{p-1}(v') = \xi b u, \beta P^0( v') = \xi' b u$ for some
$\xi, \xi' \in \mathbb{F}_{p^n}$.
For any
$e \in \mathbb{F}_{p^n}$, one gets: 
\begin{align*} d_{p-1, \times}( ev') & = d_{p-1}(ev') - \beta P^0(e v') \\
& = e \xi b u  - e^p \beta P^0(v') \\
& = e \xi b u - e^p \xi' b u = (e \xi - e^p \xi') b u.
\end{align*}
Now we know that $\xi \neq 0$, so regardless of the value of $\xi' $, there will be
\emph{at most} $p$ choices of $e$ that make the above vanish. 
This gives an upper bound of $\mathbb{Z}/p$ for the $H^1$.

However, the Hopkins--Miller calculations (\Cref{prop:e2cals}) also give a lower bound of 
$\mathbb{Z}/p$ for the Picard group: in fact, it shows that $\omega  = \pi_2 E_n$ has order
$p$ in the algebraic Picard group as $H^*(C_p, \omega^{\otimes j})
\neq H^*( C_p, (E_n)_0)$ for $j < p$. This completes the proof. 
\end{proof} 

We emphasize that, in the above proof, the additive differentials were
$\mathbb{F}_{p^n}$-linear, but the multiplicative differentials were
\emph{not}. This enables us to reduce  a group that was a priori
$\mathbb{F}_{p^n}$ to a $\mathbb{Z}/p$.
Also, one can directly show that $\xi' \neq 0$, although we do not need to do
so here. 
The fact that the difference between the additive and
multiplicative differentials is Frobenius-semilinear is all that was used, not
any of the explicit formulas. 
\section{The topological comparison of differentials}

\subsection{Context}
We begin by reviewing the context from \cite[Sec.~5 and 6]{mathew_stojanoska_picard}. 
Let $R^\bullet$ be a cosimplicial $\einfty$-ring.
In this case, there is a Bousfield--Kan spectral sequence
\begin{equation} \label{generalTotss} \pi^s \pi_t R^\bullet \implies \pi_{t-s}
\mathrm{Tot} R^\bullet.  \end{equation}
Applying the functor $\gl_1$, one also obtains a Bousfield--Kan spectral sequence
\begin{equation} \label{generalPicss} \pi^s \pi_t \gl_1( R^\bullet) \implies \pi_{t-s} \mathrm{Tot}(\gl_1(
R^\bullet)). 
\end{equation} 
The latter spectral sequence is relevant to Picard group
computations via descent theory, 
but the former spectral sequence \eqref{generalTotss} is usually better understood (e.g., because of
the multiplicative structure). 

For $t - s> 0$, these homotopy groups are determined by the Bousfield--Kan
spectral sequences for the \emph{spaces} $\Omega^\infty R^\bullet$ and
$\Omega^\infty\gl_1( R^\bullet)$, which have isomorphic connected components. 
As a result, the spectral sequences \eqref{generalTotss} and \eqref{generalPicss} can be compared in
this range.
However, it is the $t-s  = -1$ component of \eqref{generalPicss} that is
relevant to the Picard group calculations. Although at $E_2$ both spectral
sequences (except for $ s = 0, 1$) are the same in this column, the differentials
vary. 
Nonetheless, in \cite{mathew_stojanoska_picard}, the second and third authors showed that there is a
\emph{stable range} in which differentials can be compared, even for $t - s =
-1$, and the first differential beyond the stable range can be computed. 
In particular, one observes that in any range $[n, 2n-1]$, there is a functorial
equivalence of \emph{spectra}
\begin{equation} \label{squarezero} \tau_{[n, 2n-1]} R \simeq \tau_{[n, 2n-1]} \gl_1(R),
\end{equation} 
for an $\einfty$-ring $R$. This gives the stable range, and the first
differential beyond this range can be computed via reduction to a universal
example.

In this section, we will prove that after arithmetic localization 
there is a larger range in which one can compare differentials in the two
spectral sequences, analogous to the range that one obtains in algebra using
the truncated logarithm (cf.~section~\ref{algTruncLogsec}). It is 
possible to prove an analog of \eqref{squarezero} in this context (and
construct a ``spectral'' truncated logarithm); this
will be explored in forthcoming work \cite{CHMS}.

Let $R^\bullet$ be a cosimplicial $\einfty$-ring whose members are all
$\mathbb{Z}[1/(p-1)!]$-local. 
Our main goal is to demonstrate the following:  

\begin{theorem} 
\label{topunivformula}
Suppose $R^\bullet$ is a cosimplicial 
$\mathbb{Z}[1/(p-1)!]$-local
$\einfty$-ring. 
Denote by $E_*^{s,t}$ and $E_{*, \times}^{s,t}$ the associated spectral sequences 
\eqref{generalTotss} and \eqref{generalPicss} above. Suppose given $t \geq 2$
and an element $x \in E^{t+1, t}_2$ such that $x$ survives to the
$E_k$-page.
Then: 
\begin{enumerate}
\item If $k \leq  (p-1)t$, then $x$ survives to $E_{k, \times}$ in the
$\gl_1$ spectral sequence and $d_{k ,\times}(x)  = d_k(x) $.
\item
If $t$ is even and $k = (p-1)t + 1$, then $d_{k, \times}(x)  = d_k(x) + \zeta
\beta \mathcal{P}^{t/2}(x) $ for some $\zeta \in \mathbb{F}_{p}^{\times}$
(independent of $R^\bullet$). 
\end{enumerate}
\end{theorem} 

We will discuss the universal operation $\beta \mathcal{P}^{t/2}$ as
well: it arises from the right derived functors of the symmetric algebra
functor in the sense of Priddy \cite{priddy}, or as a power operation for
$\einfty$-ring spectra over $\mathbb{Z}[1/(p-1)!]$.
Again, the only fact we will need about it is its Frobenius-semilinearity
property, and we will not try to determine the constant $\zeta$. 
When $p =2 $, this recovers the comparison results of \cite[Sec.
5--6]{mathew_stojanoska_picard}. 
In particular, we shall assume throughout that $p > 2$. 
Unlike in \cite{mathew_stojanoska_picard}, we will 
use the same notation 
for elements of the multiplicative and additive spectral sequences. 

\subsection{Proofs}
\newcommand{\lan}{\mathrm{Lan}}
\Cref{topunivformula} is proved using the same strategy as in
\cite[Sec.~6]{mathew_stojanoska_picard}, although we will not try to determine $\zeta$. 
Namely, we consider the ``universal'' example of a cosimplicial $\einfty$-ring
$R^\bullet$
(with $(p-1)!$ inverted) together with 
a class in $E_k^{t+1, t}$ and prove the desired claims here, where it will
follow by a sparsity argument. 
We need to look for the cosimplicial $\einfty$-ring (which is
$\mathbb{Z}[1/(p-1)!]$-local) which corepresents the
functor
\[ X^\bullet \mapsto \Omega^\infty \left( \Sigma^{-1} \mathrm{fib}(
\mathrm{Tot}_{t+k}(X^\bullet) \to \mathrm{Tot}_t(X^\bullet) ) \right), \]
since a class in $E_k^{t+1,t}$ represents a class in $\pi_{-1}
\mathrm{Tot}_{t+k} (X^\bullet)$ trivialized in $\mathrm{Tot}_t (X^\bullet)$.

In this subsection, everything will be $\mathbb{Z}[1/(p-1)!]$-localized. 

\begin{cons}
To construct this universal cosimplicial $\einfty$-ring $R^\bullet$, we 
let $\lan$ denote the functor of left Kan extension. 
We define the pointed cosimplicial space $\mathcal{F}^\bullet$ via the homotopy pushout
\[  
\xymatrix{
\lan_{\Delta^{\leq t} \to \Delta}(\ast)_+ \ar[d] \ar[r] & \ast \ar[d] \\
\lan_{\Delta^{\leq t   + k} \to \Delta}(\ast)_+  \ar[r] &  \mathcal{F}^\bullet
}.
\]
The cosimplicial $\einfty$-ring $R^\bullet$ is defined via
\[ R^\bullet  = \mathrm{Free}( \Sigma^{-1} \Sigma^\infty
\mathcal{F}^\bullet) [1/(p-1)!].  \]
We write $\mathcal{G}^\bullet$ for the cosimplicial spectrum 
$\Sigma^{-1} \Sigma^\infty
\mathcal{F}^\bullet[1/(p-1)!]$, so that 
\begin{equation} R^\bullet  = \mathrm{Free}( \mathcal{G}^\bullet)[1/(p-1)!].
\end{equation}
\end{cons}

Pointwise, the left Kan extensions have the homotopy types of wedges of
spheres. 
As in \cite[Sec.~6.2]{mathew_stojanoska_picard}, we find 
the following properties of the pointed cosimplicial space $\mathcal{F}^\bullet: \Delta
\to \mathcal{S}_*$: 
\begin{enumerate}
\item  $\mathcal{F}(T)$ (for $T \in \Delta$) is a wedge of $S^{t+1}$'s and
$S^{t + k}$'s.
\item $\mathcal{F}^\bullet|_{\Delta^{\leq t}}$ is contractible. 
For any $T \in \Delta^{\leq t + k}$, $\mathcal{F}(T)$ is a wedge of copies
of $S^{t+1}$.
\end{enumerate}
By desuspending, this of course determines the pointwise homotopy type of
$\mathcal{G}(T)$ for any $T \in \Delta$.

The Bousfield--Kan spectral sequence for the cosimplicial spectrum
$\mathcal{F}^\bullet$ is determined as in \cite[Sec.~6.4]{mathew_stojanoska_picard} (cf.~the proof of Prop.~6.3.1,
loc.~cit.), and in particular follows from the 
calculation there of the BKSS for the cosimplicial spectrum $\Sigma^\infty_+
\mathrm{Lan}_{\Delta^{\leq t} \to \Delta}(\ast)$.
The BKSS for $\mathcal{G}^\bullet$ is a shift by one of that of
$\Sigma^\infty \mathcal{F}^\bullet[1/(p-1)!]$, and we describe
it next.  For the reader's convenience, we also include an example
in Figure~\ref{BKSSG} below. 
\begin{prop}
\label{GBKSS}
\begin{enumerate}
\item 
As graded modules over $\pi_*(S^0)[1/(p-1)!]$, 
we have
\[ H^s( \pi_*(\mathcal{G})^\bullet) \simeq  
\begin{cases} 
0 & s \neq t + 1, t + k + 1 \\
\pi_*(S^t)[1/(p-1)!] &  s = t+1 \\
\pi_*(S^{t+ k-1})[1/(p-1)!] & s = t + k + 1
 \end{cases} .
\]
\item
We let $\iota \in E_2^{t+1, t}$ be the fundamental generating class. 
There is a $d_k$ in the BKSS, which carries $\iota$ to a
generator in $E_2^{t+k+1, t+k-1} \simeq \pi_{t + k -1}(S^{t + k-1})[1/(p-1)!]$.
\end{enumerate}
\end{prop}

Now we consider the BKSS for $R^\bullet$ in the desired range, which depends
only on the Postnikov section $\tau_{\leq pt} R^\bullet$. We observe that 
that there is an equivalence of cosimplicial spectra
\[ \tau_{\leq pt} R^\bullet \simeq S^0[1/(p-1)!] \vee \mathcal{G}^\bullet \vee
\mathcal{G}^{\bullet \wedge 2}_{h \Sigma_2} \vee \dots \vee 
\mathcal{G}^{\bullet \wedge p}_{h \Sigma_p}.
\]

We claim that in the range of concern, none of the terms besides
$\mathcal{G}^\bullet$ and 
$\mathcal{G}^{\bullet \wedge p}_{h \Sigma_p}$ can contribute; 
this follows from the next lemma.  
\begin{lemma} 
\label{extpower}
For any $2 \leq u \leq p-1$, the BKSS for $(\mathcal{G}^{\bullet \wedge
u})_{h \Sigma_u}$ is zero at $E_2$ in cohomological degrees  below $(u-1)
(t+1) + (t + k + 1)$.
\end{lemma} 
\begin{proof}

Consider the cosimplicial $\pi_*(S^0)[1/(p-1)!]$-module
$M^\bullet = \pi_*\mathcal{G}^\bullet$, which is levelwise free. The cosimplicial 
module $\pi_*(\mathcal{G}^{\bullet \wedge u}_{h \Sigma_u})$ 
is obtained 
as $(M^{\otimes u})_{\Sigma_u}$ where the tensor product
is taken over $\pi_*(S^0)$, and where $\Sigma_u$ acts by permuting the factors
with appropriate signs. 
Since the order of $\Sigma_u$ is invertible, the operation of taking coinvariants is
well-behaved (and comes from the induced operation on cohomology); we obtain
that 
\[ H^*( \pi_*(\mathcal{G}^{\bullet \wedge u}_{h \Sigma_u}))
\simeq \left( H^*(M^\bullet)^{\otimes u}  \right)_{\Sigma_u}.
\]

For $s \leq (u-1)(t+1) +  (t + k+1)$, the only contribution to $H^*(  \pi_*(
\mathcal{G}^{\bullet \wedge u}))$ is in cohomological degree $u (t+1)$ 
and is a free module of rank one over $\pi_* (S^0)[1/(p-1)!]$.
However, the $\Sigma_u$ action is by the sign representation which annihilates
this upon passage to the coinvariants. 
\end{proof} 

Finally, we need to determine the contribution of $\mathcal{G}^{\bullet \wedge
p}_{h \Sigma_p}$ in case $k = (p-1)t  + 1$. 
Here we assume $t$ \emph{even.}
In the range in question, this only contributes via its
$\pi_{pt}$, given by 
\[ 
\pi_{pt} (\mathcal{G}^{\bullet \wedge pt}_{h \Sigma_p}) \simeq \mathrm{Sym}^p \pi_t
\mathcal{G}^{\bullet}.
\]

\begin{proof}[Proof of \Cref{topunivformula}]
Following the outline of \cite{mathew_stojanoska_picard}, it suffices to determine the behavior of the
tautological class in the BKSS for $\gl_1(R^\bullet)$. 
In the BKSS for $R^\bullet$, the tautological class supports a $d_k$
as in \Cref{GBKSS} (by naturality for the map $\mathcal{G}^\bullet \to R^\bullet$). 
We observe that the
tautological class in the \emph{multiplicative} spectral sequence cannot support a differential 
$d_{i, \times}$ for $i < k$ as there is no suitable target in the spectral sequence. 
Indeed, for connectivity reasons, there are only the terms $\mathcal{G}^{\bullet \wedge i}_{h \Sigma_i}$
for $i < p$ that can contribute here, and these do nothing 
thanks to \Cref{extpower}. 

If $k \leq (p-1) t $, we find that the target for a
$d_{k ,\times}$ in the $\gl_1$ spectral sequence is simply a
$\mathbb{Z}[1/(p-1)!]$; there are no additional contributions.
As in \cite{mathew_stojanoska_picard}, we conclude (by comparing with
square-zero extensions) that we can identify the differentials $d_k$ and $d_{k,
\times}$. 

Consider, finally, the case $k = (p-1) t + 1$ and $t$ even. In this case, the target of the
differential $d_{k, \times}$ still contains 
$\mathbb{Z}[1/(p-1)!] d_k \iota$, but also
has another factor as $H^{pt + 2}( \pi_{pt}  R^\bullet)$, which we need to
determine. 

Here we will need to use some of the theory of \emph{right derived functors} of
symmetric powers \cite{priddy}. 
They determine the cohomology of the symmetric powers of a cosimplicial object. 
The next proposition follows from \cite[Sec.~3 and 4]{priddy}, and in
particular Theorem 4.0.1 in loc.~cit., by running a Bockstein spectral
sequence. We note also that the work of Kraines \cite{Kra66} relates these
operations to $p$-fold Massey powers. 
\begin{prop} 
Let $X^\bullet$ be a cosimplicial $\mathbb{Z}[1/(p-1)!]$-module which is
levelwise free. 
Suppose $H^*(X^\bullet) $ is concentrated in degree $k$, $k$ odd, and 
$H^k( X^\bullet) \simeq \mathbb{Z}[1/(p-1)!]$ generated by a tautological
class $\iota$.
Then 
for each $i  \leq (k-1)/2$, we have
$H^{k+2i(p-1) + 1}( \sym^p X^\bullet) \simeq \mathbb{Z}/p$ generated by a class $\beta
\mathcal{P}^i \iota$. 
\end{prop}

Now we take $k = t + 1$ in the above. It follows that 
$H^{pt + 2}( \pi_{pt}( \mathcal{G}^{\bullet \wedge pt}_{h \Sigma_p})) \simeq
\mathbb{Z}/p$, generated by $\beta \mathcal{P}^{t/2} \iota$. It follows that we
must have  a ``universal'' formula
\[ d_{k, \times}(\iota) = d_k(\iota) + \zeta \beta \mathcal{P}^{t/2} \iota,  \]
where $\zeta \in \mathbb{F}_p$. We note that the coefficient of $d_k(\iota)$ is
argued to be one similarly as in \cite{mathew_stojanoska_picard}.

To complete the argument, we need to argue that $\zeta \neq 0$. 
To do this, we consider (for any even $t$) an
$\einfty$-ring $A$ under $\mathbb{F}_p$ such that
$\pi_*(A) \simeq \mathbb{F}_p[u]/(u^{p+1})$  for $|u| = t$. 
We can construct $A$ by taking the free $\einfty$-ring over $\mathbb{F}_p$ on a
class in degree $t$ and truncating at $pt$, i.e.,
\[ A = \tau_{\leq pt} \mathrm{Free}_{\mathbb{F}_p}( \Sigma^t \mathbb{F}_p),  \]
where $\mathrm{Free}_{\mathbb{F}_p}$ denotes the free $\einfty$-algebra functor
over $\mathbb{F}_p$.
In \Cref{notdeloop} below, we will show that the natural map 
\begin{equation} \label{notloopmap} K( \mathbb{F}_p, t) \to GL_1( A),  \end{equation}
inducing an isomorphism on $\pi_t$, does not deloop. 
Of course, the analogous map $K( \mathbb{F}_p, t) \to \Omega^\infty A$ does
deloop. 
Equivalently, the associated (Atiyah--Hirzebruch) spectral sequence for $\gl_1(A)^{K(
\mathbb{F}_p, t+1)}$ 
admits a differential on the tautological class
in $H^{t+1}( K(\mathbb{F}_p, t+1), \pi_t A)$. 
By interpreting the AHSS as a BKSS, we find that there must be a nonzero $d_{k,
\times}$ and therefore $\zeta \neq 0$. 
\end{proof}

\begin{lemma}
\label{notdeloop}
The map \eqref{notloopmap} does not deloop to a map $K( \mathbb{F}_p, t+1)
\to BGL_1(A)$. 
\end{lemma}

\begin{proof} 

To see that 
\eqref{notloopmap} is not a loop map, we use the universal properties of
$GL_1$ (cf. \cite[Sec.~3]{ABGHR}). 
If \eqref{notloopmap} were a loop map, we would have a morphism of associative
ring spectra
over $\mathbb{F}_p$
\[ \mathbb{F}_p \wedge \Sigma^\infty_+ K( \mathbb{F}_p, t)  \to A, \]
inducing an isomorphism on $\pi_t$. However, the class in $\pi_t A$ has nonzero
$p$th power, while the $p$th power of the tautological class in 
$\mathbb{F}_p \wedge \Sigma^\infty_+ K( \mathbb{F}_p, t) $
vanishes by \cite[Lemma 6.1]{CLM76}.\footnote{We are grateful to Tyler Lawson for showing us this argument
at $p = 2$, which was used in \cite{mathew_stojanoska_picard}.} 

\end{proof}

\begin{figure}
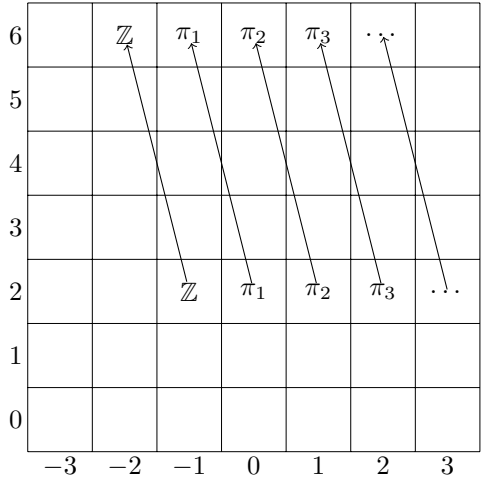
 
\begin{sseq}[entrysize=8.5mm,labelstep=1, arrows=to]{-3...3}{0...6}

\ssmoveto{-1}{2}
\ssdrop{\mathbb{Z}}
\ssname{a1}

\ssmoveto{0}{2}
\ssdrop{\pi_1}
\ssname{b1}

\ssmoveto{1}{2}
\ssdrop{\pi_2}
\ssname{c1}

\ssmoveto{2}{2}
\ssdrop{\pi_3}
\ssname{d1}

\ssmoveto{3}{2}
\ssdrop{\dots}
\ssname{e1}

\ssmoveto{-2}{6}
\ssdrop{\mathbb{Z}}
\ssname{a2}

\ssmoveto{-1}{6}
\ssdrop{\pi_1}
\ssname{b2}

\ssmoveto{0}{6}
\ssdrop{\pi_2}
\ssname{c2}

\ssmoveto{1}{6}
\ssdrop{\pi_3}
\ssname{d2}

\ssmoveto{2}{6}
\ssdrop{\dots}
\ssname{e2}

\ssgoto{a1} \ssgoto{a2} \ssstroke \ssarrowhead
\ssgoto{b1} \ssgoto{b2} \ssstroke \ssarrowhead
\ssgoto{c1} \ssgoto{c2} \ssstroke \ssarrowhead
\ssgoto{d1} \ssgoto{d2} \ssstroke \ssarrowhead
\ssgoto{e1} \ssgoto{e2} \ssstroke \ssarrowhead
\end{sseq}
\caption{The BKSS for $\mathcal{G}^\bullet$ for $t = 1, \, k = 4$} 
\label{BKSSG}
\end{figure}


\FloatBarrier
\bibliography{biblio}\bibliographystyle{alpha}
\end{document}